\documentclass[11pt,oneside]{article}

\usepackage[a4paper,top=28mm,bottom=28mm,left=30mm,right=30mm]{geometry}
\usepackage{setspace}
\usepackage[T1]{fontenc}
\usepackage[utf8]{inputenc}
\usepackage{lmodern}
\usepackage{microtype}

\usepackage{amsmath,amssymb,mathtools}
\usepackage{bm}
\usepackage{amsthm}
\numberwithin{equation}{section}

\usepackage{graphicx}
\usepackage{booktabs}
\usepackage{array}
\usepackage{tabularx}
\usepackage{longtable}
\usepackage{multirow}
\usepackage{caption}
\usepackage{subcaption}
\usepackage{placeins}
\usepackage{enumitem}
\usepackage{xcolor}
\setlist{nosep,leftmargin=2.2em}

\usepackage[numbers,sort&compress,square]{natbib}
\usepackage[hidelinks]{hyperref}
\usepackage[nameinlink,noabbrev]{cleveref}

\crefname{equation}{Eq.}{Eqs.}
\crefname{figure}{Fig.}{Figs.}
\crefname{table}{Table}{Tables}
\crefname{theorem}{Theorem}{Theorems}
\crefname{lemma}{Lemma}{Lemmas}
\crefname{corollary}{Corollary}{Corollaries}
\crefname{remark}{Remark}{Remarks}
\crefname{section}{Section}{Sections}
\crefname{subsection}{Section}{Sections}

\theoremstyle{plain}
\newtheorem{theorem}{Theorem}[section]
\newtheorem{corollary}[theorem]{Corollary}

\theoremstyle{remark}
\newtheorem{remark}[theorem]{Remark}

\newcommand{\bk}{\bm{k}}
\newcommand{\bq}{\bm{q}}
\newcommand{\br}{\bm{r}}
\newcommand{\ba}[1]{\bm{a}_{#1}}

\newcommand{\voigtorder}{[11,22,33,23,13,12]}
\newcommand{\engshear}{engineering shear convention \((\gamma_{ij}=2\varepsilon_{ij})\)}

\title{\bfseries An Efficient Parity-Blocked Method for Band-Structure Computation of 3D Anisotropic Phononic Crystals}

\author{%
\parbox{0.96\textwidth}{\centering
Jingkai Zhang\textsuperscript{1}\quad
Xing-Long Lyu\textsuperscript{2,*}\quad
Tiexiang Li\textsuperscript{1,3}\quad
Wen-Wei Lin\textsuperscript{3,4}\\[0.55em]
\footnotesize
\textsuperscript{1}School of Mathematics and Shing-Tung Yau Center, Southeast University,\\
Nanjing 211189, People's Republic of China\\
\textsuperscript{2}School of Mathematical Sciences, Nanjing Normal University, Nanjing 210023, People's Republic of China\\
\textsuperscript{3}Shanghai Institute for Mathematics and Interdisciplinary Sciences (SIMIS),\\
Shanghai 200433, People's Republic of China\\
\textsuperscript{4}Department of Applied Mathematics, National Yang Ming Chiao Tung University, Hsinchu 300, Taiwan\\[0.55em]
\texttt{jkzhangmath@seu.edu.cn} (J. Zhang); \texttt{xllyu@njnu.edu.cn} (X.-L. Lyu)\\
\texttt{txli@seu.edu.cn} (T. Li); \texttt{wwlin@outlook.com} (W.-W. Lin)\\[0.25em]
\textsuperscript{*}Corresponding author
}}

\date{}

\begin{document}
\maketitle

\begin{abstract}
Band-structure calculations for three-dimensional anisotropic phononic crystals require the repeated solution of large elastic generalized eigenvalue problems along Bloch paths. In standard staggered-grid discretizations, anisotropic coupling may involve derivative components located at incompatible grid positions, so additional interpolation or averaging closures are often introduced. This paper proposes a parity-blocked rotated staggered discretization based on four Bloch-periodic body-diagonal differences. The directional derivatives are reconstructed from these diagonal differences, leading to a Hermitian $B_hC_hB_h^H$ generalized eigenvalue formulation that incorporates anisotropic derivative coupling without separate interpolation closures. On even grids, when the stiffness and mass matrices are nodewise local multiplication matrices, the body-diagonal shifts preserve two independent parity invariants. The discrete velocity space is then decomposed exactly into four mutually independent block subspaces, and the full discrete spectrum can be recovered by solving the four smaller eigenvalue problems and merging their spectra. The full and block formulations are further organized in a unified Fourier SVD framework, which supports $\Gamma$-point zero-mode treatment, shift-invert Krylov iteration, inner PCG solves, and GPU matrix-vector products. Numerical experiments for a three-dimensional two-phase anisotropic phononic crystal show that the block implementation preserves the full-space spectrum while substantially reducing the wall-clock time. The results demonstrate that the proposed method provides a structured and efficient solver for large-scale band-structure computations of three-dimensional anisotropic phononic crystals.
\end{abstract}


\bigskip

\section{Introduction}
\label{sec1}

Phononic crystals regulate the propagation of elastic waves through periodic modulation of material parameters. Their band gaps, guided modes, localization effects, and directional wave propagation have become central topics in wave functional materials. Since the early work of Sigalas and Economou and of Kushwaha et al. on band structures of periodic elastic composites, the theory, numerical simulation, and structural design of phononic crystals have developed rapidly \cite{SigalasEconomou1992,Kushwaha1993,Hussein2014}. In three dimensions, band structure calculation amounts to solving, for each Bloch wave vector \(\bk\), an elastic eigenvalue problem on a periodic unit cell subject to Bloch periodic boundary conditions, and then tracking the low-frequency spectral branches along a prescribed path in the first Brillouin zone.

Existing approaches include plane wave expansion, finite differences, finite elements, time-domain methods, and complex Bloch wave formulations \cite{SigalasEconomou1992,Kushwaha1993,Laude2009,LyuTianLiLin2024JSC}. When material anisotropy, piezoelectric coupling, or multiphysical coupling is significant, elastic stiffness couplings may strongly affect dispersion relations, modal patterns, and band gap formation \cite{Wu2004PRB,WuHsuHuang2005}. Finite element frameworks for anisotropic, piezoelectric, and multicomponent elastic phononic crystals have also been developed in recent years \cite{Wang2021APM,Yan2025JCP}. For 3D phononic crystals with both general anisotropy and spatially heterogeneous material distributions, however, the computational difficulty is not only the number of degrees of freedom, but also the compatibility among Bloch periodicity, anisotropic constitutive coupling, and the layout of discrete variable spaces.

A typical difficulty is the lack of collocation in standard staggered discretizations. In the Voigt representation of the linear elastic equations, a general anisotropic stiffness matrix may contain many nonzero off-diagonal coupling entries. A stress component may depend not only on the normal strains \(\varepsilon_{11},\varepsilon_{22},\varepsilon_{33}\), but also on the shear strains \(\gamma_{23},\gamma_{13},\gamma_{12}\). In a standard velocity-stress staggered grid, different velocity components, stress components, and their derivatives are placed at different spatial locations. Hence, some derivative terms required to construct a target stress component do not lie at the same collocation points as that stress component. Although interpolation or averaging closures can make the scheme computable, they introduce additional discretization choices and weaken the structural consistency of the discrete operator.

Related issues have been studied systematically in finite difference simulations of elastic waves. The velocity-stress staggered grid schemes of Virieux \cite{Virieux1984,Virieux1986} form a classical framework for elastic wave propagation. The rotated staggered grid methods developed by Saenger and coauthors \cite{Saenger2000,Saenger2004,BansalSen2008} provide an important strategy for variable placement in complex and anisotropic elastic media. Subsequent studies have shown that grid layout, derivative reconstruction, and finite difference symbols influence numerical dispersion, stability, and long-time propagation accuracy in anisotropic elastic wave simulations \cite{Bernth2011,Yang2015,Gao2017}. Their combination with absorbing boundary conditions has also been investigated \cite{Chen2006}. In addition, collocated discretizations and Lebedev-type highly symmetric staggered schemes can be used for 3D anisotropic elastic wave problems \cite{Zhang2012,Lisitsa2010}; nevertheless, when these methods are applied to Bloch periodic eigenvalue problems, the consistency among numerical modes, physical modes, and discrete symmetries still requires careful treatment \cite{Koene2021}.

Unlike the above time-domain propagation schemes, this work focuses on frequency-domain Bloch periodic eigenvalue problems. The central question is how to construct a structured discrete operator that is simultaneously compatible with general anisotropic constitutive coupling, Bloch phase periodicity, and large-scale spectral computation. To this end, we first construct phase-shifted periodic differences along four body-diagonal directions and then reconstruct the three lattice coordinate derivatives by linear combinations. The resulting operator has the Hermitian generalized eigenvalue form \(B_hC_hB_h^H v_h=\omega^2M_hv_h\), so that the divergence and strain adjoint relation is preserved after discretization. This construction incorporates the derivative couplings of a general Voigt stiffness matrix into one operator framework and avoids closing the anisotropic stress terms by separate interpolation formulas. The body-diagonal shifts have an additional algebraic consequence on even grids: they preserve two independent parity labels. The discrete velocity space can therefore be decomposed into four mutually uncoupled invariant subspaces, and the full discrete spectrum can be recovered by solving the four block problems independently and merging their eigenvalues.

The main contributions of this work are threefold. First, for 3D Bloch periodic elastic eigenvalue problems with general anisotropy, we construct a rotated staggered derivative reconstruction based on four body-diagonal Bloch shifts and obtain a unified \(B_hC_hB_h^H\) discrete structure. Second, we prove the even-grid parity-block invariance, state the assumptions under which it is exact, and derive the corresponding reduced generalized eigenvalue problems. Third, we formulate the full-space and parity-block spaces in a common Fourier SVD representation and combine this representation with \(\Gamma\)-point nullspace treatment, shift-invert Lanczos iteration, inner PCG solves, and GPU matrix-vector products. In this formulation, Fourier SVD is used as the organizing and acceleration framework for the derivative part, whereas the new algebraic reduction comes from the parity invariance induced by the body-diagonal differences.

The remainder of this paper is organized as follows. Section~\ref{sec:form} formulates the elastic eigenvalue problem. Section~\ref{sec:blochdisc} constructs the Bloch periodic rotated staggered discretization. Section~\ref{sec:block} proves the parity blocking on even grids. Section~\ref{sec:block_wsvd} develops the Fourier SVD block reduction. Section~\ref{sec:impl} describes the discrete spectral solver, including the Lanczos iteration, matrix-vector products, and \(\Gamma\)-point deflation. Section~\ref{sec:num} presents the numerical experiments. Section~\ref{sec:conclusion} concludes the paper.

\section{Elastic Eigenvalue Formulation}
\label{sec:form}

We consider the frequency-domain linear elastic wave equation in a 3D anisotropic phononic crystal,
\begin{equation}\label{eq:elastic_freq_main}
-\omega^2\rho(\mathbf x)\,\bm u(\mathbf x)=\nabla\cdot\bm\tau(\mathbf x),
\end{equation}
where \(\mathbf x\) denotes the spatial position, \(\omega\) is the angular frequency, \(\rho(\mathbf x)\) is the density, \(\bm u(\mathbf x)\) is the displacement field, and \(\bm\tau(\mathbf x)\) is the Cauchy stress tensor. Under the small strain assumption,
\begin{equation}\label{eq:strain_def_main}
\bm\varepsilon(\bm u)=\frac12\bigl(\nabla\bm u+(\nabla\bm u)^\top\bigr),
\end{equation}
and the constitutive relation is
\begin{equation}\label{eq:constitutive_4th_main}
\tau_{ij}=C_{ijkl}\,\varepsilon_{kl}.
\end{equation}
Since both stress and strain are symmetric tensors, the fourth order elastic tensor can be represented by a symmetric \(6\times6\) matrix in Voigt notation. The Voigt components used in this paper are always formed in a fixed orthonormal Cartesian frame \((x_1,x_2,x_3)\). We use the Voigt ordering \voigtorder\ and the \engshear. Define
\begin{equation}\label{eq:tau_voigt_main}
\widetilde{\bm\tau}:=
[\tau_{x_1x_1},\tau_{x_2x_2},\tau_{x_3x_3},
\tau_{x_2x_3},\tau_{x_1x_3},\tau_{x_1x_2}]^\top,
\end{equation}
and
\begin{equation}\label{eq:eps_voigt_main}
\widetilde{\bm\varepsilon}(\bm u):=
[\varepsilon_{x_1x_1},\varepsilon_{x_2x_2},\varepsilon_{x_3x_3},
2\varepsilon_{x_2x_3},2\varepsilon_{x_1x_3},2\varepsilon_{x_1x_2}]^\top.
\end{equation}
The constitutive law is then written as
\begin{equation}\label{eq:constitutive_voigt_main}
\widetilde{\bm\tau}(\mathbf x)=\bm C(\mathbf x)\,\widetilde{\bm\varepsilon}(\bm u(\mathbf x)),
\qquad
\bm C(\mathbf x)\in\mathbb R^{6\times6}.
\end{equation}

When a nonorthogonal primitive lattice basis is used for grid indexing, we distinguish the reduced lattice coordinate \(\br=(r_1,r_2,r_3)^\top\) from the physical Cartesian coordinate \(\mathbf x\). If
\begin{equation}\label{eq:coordinate_map_main}
\mathbf x=A_{\rm lat}\br,
\qquad
A_{\rm lat}:=[\ba{1},\ba{2},\ba{3}],
\end{equation}
then the physical gradient is obtained from the reduced coordinate gradient by
\begin{equation}\label{eq:gradient_transform_main}
\nabla_{\mathbf x}=A_{\rm lat}^{-T}\nabla_{\br}.
\end{equation}
Thus the stiffness matrix \(\bm C\) and the Voigt strain vector are represented in the Cartesian frame, whereas the Bloch shifts and FFT indexing are naturally described in reduced lattice coordinates. In the discrete operator, the reconstructed reduced coordinate derivatives are transformed by \eqref{eq:gradient_transform_main} before they enter the Cartesian strain and divergence matrices. This convention prevents the nonorthogonality of the FCC primitive cell from being confused with the Cartesian Voigt convention.

This paper focuses on general anisotropy. In this case, off-diagonal coupling entries in the Voigt stiffness matrix, such as \(C_{14},C_{15},C_{16}\), are generally nonzero. These entries make a stress component depend simultaneously on several normal and shear strain components. Consequently, in a standard staggered-grid layout, several derivative terms required for a target stress component may not lie at the same collocation points as that stress component. This lack of collocation is the main motivation for introducing the rotated staggered derivative reconstruction.

For the subsequent discretization and eigenvalue solution, we introduce the velocity type variable
\begin{equation}\label{eq:v_def_main}
\bm v:=\mathrm i\omega\,\bm u.
\end{equation}
The momentum equation becomes
\begin{equation}\label{eq:momentum_v_main}
\mathrm i\omega\,\rho(\mathbf x)\,\bm v(\mathbf x)=\nabla\cdot\bm\tau(\mathbf x).
\end{equation}
To obtain a Hermitian spectral structure, define the divergence type differential operator
\begin{equation}\label{eq:B_op_main}
B=
\begin{bmatrix}
-\mathrm{i}\partial_{x_1}&0&0&0&-\mathrm{i}\partial_{x_3}&-\mathrm{i}\partial_{x_2}\\
0&-\mathrm{i}\partial_{x_2}&0&-\mathrm{i}\partial_{x_3}&0&-\mathrm{i}\partial_{x_1}\\
0&0&-\mathrm{i}\partial_{x_3}&-\mathrm{i}\partial_{x_2}&-\mathrm{i}\partial_{x_1}&0
\end{bmatrix}.
\end{equation}
In our notation, \(B\) is the divergence type operator mapping Voigt stress variables to the momentum equation, \(\partial_{x_i}\) denotes the Bloch periodic physical derivative in the Cartesian direction \(x_i\), and \(B^H\) maps velocity variables to Voigt strain variables. The superscript \(H\) denotes the Hermitian transpose. The operator \(BCB^H\) therefore preserves the divergence and strain adjoint structure at the discrete level. With this convention, the frequency-domain elastic system can be written as the block eigenvalue problem
\begin{equation}\label{eq:first_order_eig_block_main}
\left[
\begin{array}{cc}
\mathbf 0 & \rho(\mathbf x)^{-1} B\\
\bm C(\mathbf x) B^H & \mathbf 0
\end{array}
\right]
\left[
\begin{array}{c}
\bm v(\mathbf x)\\
\widetilde{\bm\tau}(\mathbf x)
\end{array}
\right]
=
\omega
\left[
\begin{array}{c}
\bm v(\mathbf x)\\
\widetilde{\bm\tau}(\mathbf x)
\end{array}
\right].
\end{equation}
Eliminating the stress variable gives the continuous generalized eigenvalue model used in this paper,
\begin{equation}\label{eq:LEEP_main}
\bigl(B\,\bm C(\mathbf x)\,B^H\bigr)\bm v(\mathbf x)
=
\omega^2\rho(\mathbf x)\bm v(\mathbf x).
\end{equation}
A similar velocity-stress elimination form appears in fast solvers for 3D phononic crystal linear elastic eigenvalue problems \cite{LyuTianLiLin2024JSC}. Building on this \(BCB^H\) spectral structure, the present work develops a body-diagonal Bloch difference reconstruction for general anisotropic constitutive laws. Thus \(B\,\bm C(\mathbf x)\,B^H\) is the elastic stiffness operator acting on velocity type unknowns: \(B^H\) first maps the velocity field to Voigt strains, \(\bm C(\mathbf x)\) applies the local anisotropic constitutive law, and \(B\) then gives the corresponding divergence type momentum balance. The Bloch periodic discretization, rotated staggered derivative reconstruction, and block decomposition developed below are all based on this generalized eigenvalue framework. The componentwise correspondence between \eqref{eq:first_order_eig_block_main} and \eqref{eq:LEEP_main} is given in Appendix~\ref{app:block_verify}.

\section{Bloch Periodic Rotated Staggered Discretization}
\label{sec:blochdisc}

\subsection{Bloch Shifts on the Unit Cell}

For band structure calculations of periodic media, the role of the Bloch condition is to reduce the wave problem on the infinite periodic medium to an eigenvalue problem on a unit cell, while encoding information across cell boundaries by phase factors determined by the wave vector. Let \(\ba{1},\ba{2},\ba{3}\) be the primitive lattice vectors and let \(A_{\rm lat}=[\ba{1},\ba{2},\ba{3}]\). We use \(\bq\) for the physical Bloch wave vector and \(\bk=(\kappa_1,\kappa_2,\kappa_3)^\top\) for its reduced reciprocal coordinate, defined by
\begin{equation}\label{eq:reduced_bloch_coordinate}
\bq\cdot\ba{\ell}=2\pi\kappa_\ell,
\qquad \ell=1,2,3.
\end{equation}
Therefore the phase gained across the \(\ell\)-th primitive translation is \(\exp(\mathrm i2\pi\kappa_\ell)\). For a given reduced Bloch vector \(\bk\), the velocity and stress fields satisfy
\begin{equation}\label{eq:bloch_main}
\bm v(\mathbf x+\ba{\ell})=e^{\mathrm i\,2\pi\kappa_\ell}\,\bm v(\mathbf x),\qquad
\widetilde{\bm\tau}(\mathbf x+\ba{\ell})=e^{\mathrm i\,2\pi\kappa_\ell}\,\widetilde{\bm\tau}(\mathbf x),
\qquad \ell=1,2,3.
\end{equation}
At the discrete level, Bloch periodicity is therefore not a purely periodic boundary condition, but a periodic shift accompanied by a phase correction. A key step in the following construction is to express these phase-shifted periodic translations in an algebraic form suitable for matrix operations.

Let \(n_1,n_2,n_3\) be the numbers of grid points in the three lattice directions, and let \(n=n_1n_2n_3\) be the total number of grid points in a unit cell. Define
\begin{equation}\label{eq:phi_def_main}
\phi_\ell:=2\pi\kappa_\ell,\qquad \ell=1,2,3.
\end{equation}
The Bloch phase acquired when crossing the boundary in the \(\ell\)-th lattice direction is \(e^{\mathrm i\phi_\ell}\). The one-dimensional phase-shifted periodic shift is represented by
\begin{equation}\label{eq:K_def_main}
K_\ell=
\begin{pmatrix}
0 & I_{n_\ell-1}\\
e^{\mathrm i\phi_\ell} & 0
\end{pmatrix},
\qquad \ell=1,2,3.
\end{equation}
This matrix performs a standard forward shift inside the unit cell and multiplies by the Bloch phase factor when the index crosses the cell boundary.

Using tensor products, the one-dimensional shifts are extended to 3D shift operators on the unit cell:
\begin{equation}\label{eq:H_def_main}
H_1:=I_{n_3}\otimes I_{n_2}\otimes K_1,\qquad
H_2:=I_{n_3}\otimes K_2\otimes I_{n_1},\qquad
H_3:=K_3\otimes I_{n_2}\otimes I_{n_1}.
\end{equation}
Thus the Bloch periodic boundary condition is encoded in the three coordinate direction discrete shifts. The body-diagonal differences and rotated staggered derivative reconstruction below are built on these phase-shifted shift operators. Since \(H_1,H_2,H_3\) commute pairwise, they can be simultaneously diagonalized by a common unitary basis; the corresponding spectral decomposition, used for fast operator actions, is discussed in Appendix~\ref{app:K_spectrum}.

\subsection{Rotated Derivative Reconstruction}

For generally anisotropic materials, the Voigt constitutive matrix may contain many nonzero coupling entries. For example, the stress component \(\tau_{x_1x_1}\) may depend not only on the axial derivatives \(v_{x_1,x_1},v_{x_2,x_2},v_{x_3,x_3}\), but also on mixed derivative combinations such as \(v_{x_2,x_3}+v_{x_3,x_2}\), \(v_{x_1,x_3}+v_{x_3,x_1}\), and \(v_{x_1,x_2}+v_{x_2,x_1}\). In a standard staggered grid discretization, these derivative terms are generally located at different spatial positions from the target stress component, and additional interpolation or averaging is often needed to close the scheme. For 3D generally anisotropic Bloch periodic eigenvalue problems, such lack of collocation not only complicates implementation but also weakens the algebraic structural consistency of the discrete operator.

To reduce the effect of this lack of collocation, we do not construct all first derivatives independently along the lattice coordinate axes. Instead, we first introduce Bloch periodic differences along four body-diagonal directions and then reconstruct the three lattice coordinate derivatives by linear combinations of these body-diagonal differences. Let \(h_1,h_2,h_3\) be the grid spacings in the three lattice coordinate directions, and set
\[
h_d:=\sqrt{h_1^2+h_2^2+h_3^2}.
\]
In the lattice coordinate basis \(\mathbf e_1,\mathbf e_2,\mathbf e_3\), define four body-diagonal directions by
\begin{equation}\label{eq:d_tilde_3d_main}
\begin{aligned}
\widetilde{\mathbf d}_1&=\frac{h_1}{h_d}\mathbf e_1+\frac{h_2}{h_d}\mathbf e_2+\frac{h_3}{h_d}\mathbf e_3,\\
\widetilde{\mathbf d}_2&=\frac{h_1}{h_d}\mathbf e_1+\frac{h_2}{h_d}\mathbf e_2-\frac{h_3}{h_d}\mathbf e_3,\\
\widetilde{\mathbf d}_3&=\frac{h_1}{h_d}\mathbf e_1-\frac{h_2}{h_d}\mathbf e_2+\frac{h_3}{h_d}\mathbf e_3,\\
\widetilde{\mathbf d}_4&=\frac{h_1}{h_d}\mathbf e_1-\frac{h_2}{h_d}\mathbf e_2-\frac{h_3}{h_d}\mathbf e_3.
\end{aligned}
\end{equation}
These four directions correspond to the four body-diagonals in the 3D rotated staggered grid. After constructing difference operators along these directions, the coordinate direction derivatives are reconstructed as
\begin{equation}\label{eq:3d_deriv_recon_main}
\begin{aligned}
\frac{\partial}{\partial r_1}\bm v
&=
c_1\left(D_{\widetilde d_1}+D_{\widetilde d_2}+D_{\widetilde d_3}+D_{\widetilde d_4}\right)\bm v,\\
\frac{\partial}{\partial r_2}\bm v
&=
c_2\left(D_{\widetilde d_1}+D_{\widetilde d_2}-D_{\widetilde d_3}-D_{\widetilde d_4}\right)\bm v,\\
\frac{\partial}{\partial r_3}\bm v
&=
c_3\left(D_{\widetilde d_1}-D_{\widetilde d_2}+D_{\widetilde d_3}-D_{\widetilde d_4}\right)\bm v,
\end{aligned}
\end{equation}
where the reconstruction coefficients are determined by the grid spacings,
\begin{equation}\label{eq:ci_recon_main}
c_1=\frac{h_d}{4h_1},\qquad
c_2=\frac{h_d}{4h_2},\qquad
c_3=\frac{h_d}{4h_3}.
\end{equation}

This reconstruction expresses several coupled derivative terms, which would otherwise be distributed over different locations in a standard staggered layout, through a unified combination of body-diagonal differences. It thereby reduces the reliance on additional interpolation in general anisotropic settings and provides the basis for assembling the subsequent divergence and strain adjoint structure. Figure~\ref{fig:rot_stag_3d_main} illustrates the geometry of the 3D rotated staggered grid.

\begin{figure}[htbp]
\centering
\includegraphics[width=0.72\textwidth]{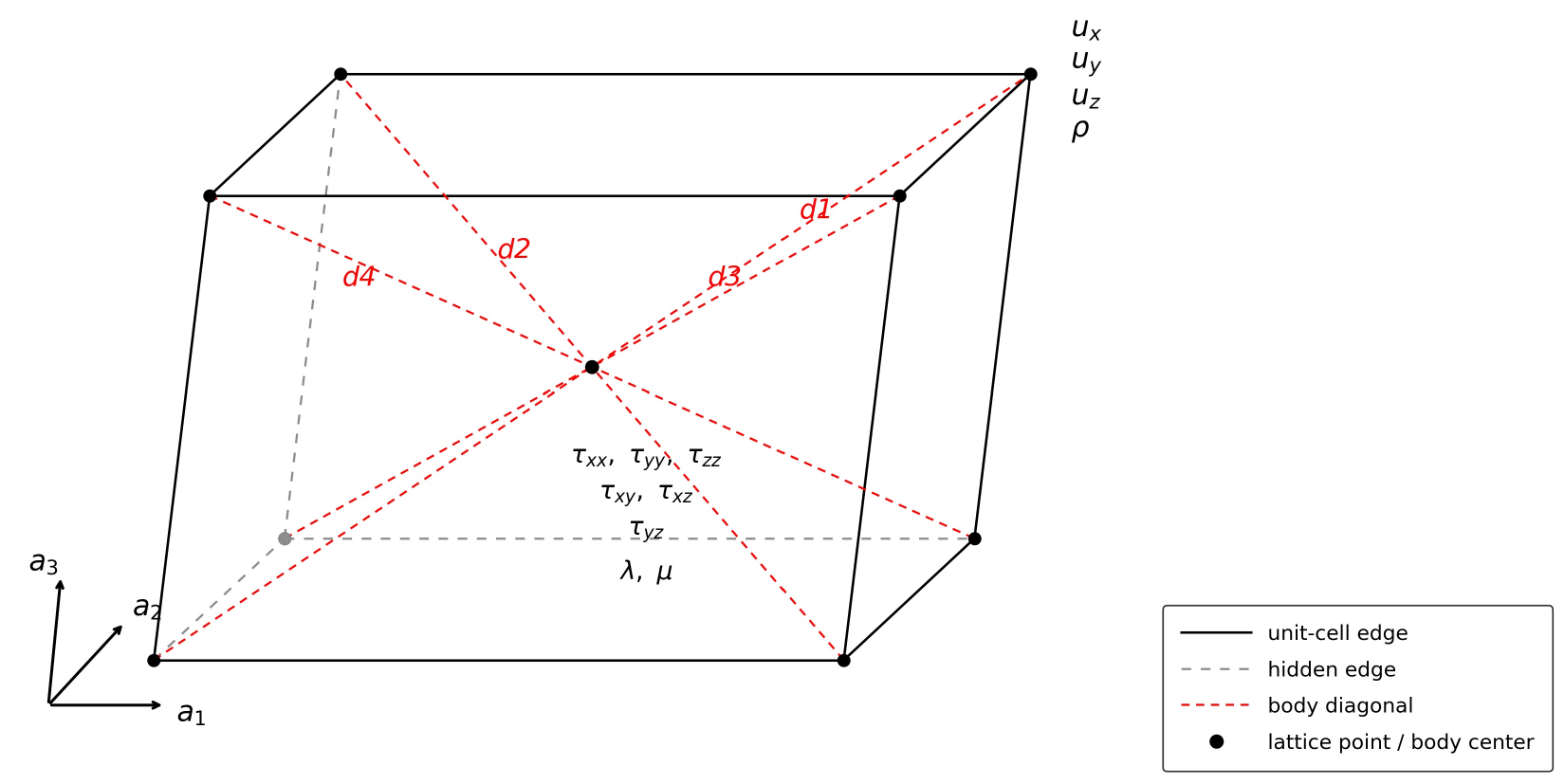}
\caption{Schematic illustration of the 3D rotated staggered grid. The solid black lines indicate the unit cell boundary, the dashed gray lines indicate hidden edges, the dashed red lines indicate the four body-diagonal difference directions, and the black dots denote grid nodes or cell center locations.}
\label{fig:rot_stag_3d_main}
\end{figure}

\subsection{Body Diagonal Differences and Assembly}

The previous subsection described the geometry of the rotated staggered derivative reconstruction. We now express it in an algebraic form suitable for matrix assembly. Since the Bloch periodic shifts in the three coordinate directions have been defined in \eqref{eq:H_def_main}, a one step translation along each body-diagonal direction can be written as a product of the basic shift operators:
\begin{equation}\label{eq:Sd_def_main}
S_{d_1}:=H_3H_2H_1,\qquad
S_{d_2}:=H_3^{-1}H_2H_1,\qquad
S_{d_3}:=H_3H_2^{-1}H_1,\qquad
S_{d_4}:=H_3^{-1}H_2^{-1}H_1.
\end{equation}
These four operators are the algebraic carriers of the four Bloch periodic body-diagonal shifts.

The forward difference along each body-diagonal direction is defined by
\begin{equation}\label{eq:Dd_def_main}
D_{d_\ell}:=\frac{S_{d_\ell}-I}{h_d},
\qquad
h_d:=\sqrt{h_1^2+h_2^2+h_3^2},
\qquad \ell=1,2,3,4,
\end{equation}
where \(h_d\) is the length of one body-diagonal step. The coordinate direction derivative reconstruction then becomes
\begin{equation}\label{eq:Dtilde_recon_main}
\begin{aligned}
\widetilde D_1&:=\frac{h_d}{4h_1}
\left(D_{d_1}+D_{d_2}+D_{d_3}+D_{d_4}\right),\\
\widetilde D_2&:=\frac{h_d}{4h_2}
\left(D_{d_1}+D_{d_2}-D_{d_3}-D_{d_4}\right),\\
\widetilde D_3&:=\frac{h_d}{4h_3}
\left(D_{d_1}-D_{d_2}+D_{d_3}-D_{d_4}\right).
\end{aligned}
\end{equation}
Thus the three coordinate direction derivatives are not constructed independently. They are obtained as linear combinations of four body-diagonal differences, which is the defining algebraic feature of the rotated staggered reconstruction.

The operators \(\widetilde D_1,\widetilde D_2,\widetilde D_3\) approximate the reduced coordinate derivatives in the grid coordinate \(\br\). To match the Cartesian operator \eqref{eq:B_op_main}, which involves \(-\mathrm i\partial_{x_m}\), first set
\begin{equation}\label{eq:Dhat_def_main}
\widehat{\mathcal D}_s:=-\mathrm i\,\widetilde D_s,
\qquad s=1,2,3.
\end{equation}
If the primitive lattice basis is nonorthogonal, the physical derivative matrices are then obtained from \eqref{eq:gradient_transform_main} as
\begin{equation}\label{eq:Dcal_cartesian_transform_main}
\mathcal D_m:=\sum_{s=1}^3 (A_{\rm lat}^{-1})_{s m}\widehat{\mathcal D}_s,
\qquad m=1,2,3.
\end{equation}
For an orthonormal computational cell, \(A_{\rm lat}=I\) and therefore \(\mathcal D_m=\widehat{\mathcal D}_m\). The discrete divergence type matrix is assembled as
\begin{equation}\label{eq:B_discrete_main}
B_h(\bk)=
\begin{bmatrix}
\mathcal D_1 & 0 & 0 & 0 & \mathcal D_3 & \mathcal D_2\\
0 & \mathcal D_2 & 0 & \mathcal D_3 & 0 & \mathcal D_1\\
0 & 0 & \mathcal D_3 & \mathcal D_2 & \mathcal D_1 & 0
\end{bmatrix}.
\end{equation}
The subscript \(h\) distinguishes the discrete matrix from the continuous operator in Section~\ref{sec:form}, and the dependence on \(\bk\) arises from the Bloch phases in the shift operators. Here \(B_h(\bk)\) maps Voigt stress variables to the discrete velocity equation, whereas \(B_h(\bk)^H\) maps the velocity unknowns to the discrete Voigt strain variables.

The continuous generalized eigenvalue problem \eqref{eq:LEEP_main} is finally discretized as
\begin{equation}\label{eq:discrete_GEP_main}
A_h(\bk)\mathbf v_h=\omega^2 M_h\mathbf v_h,
\qquad
A_h(\bk):=B_h(\bk)C_hB_h(\bk)^H .
\end{equation}
Here \(\mathbf v_h\in\mathbb C^{3n}\) is the discrete velocity unknown, \(C_h\in\mathbb C^{6n\times6n}\) is the discrete stiffness matrix, and \(M_h\in\mathbb C^{3n\times3n}\) is the discrete mass matrix. Equation~\eqref{eq:discrete_GEP_main} is the basic discrete model used for the block decomposition and numerical experiments.

\begin{remark}[Symbol consistency of the derivative reconstruction]
\label{rem:symbol_consistency}
The reconstructed derivative symbols are consistent with the reduced coordinate derivatives at low frequencies. Let \(\theta_s=h_s\xi_s\) be the phase associated with a smooth Fourier mode in the \(r_s\) direction and let \(h=\max_s h_s\), with comparable grid spacings. For example, the symbol of \(\widetilde D_1\) can be written as
\[
\delta_1(\theta)
=\frac{1}{4h_1}\sum_{\varepsilon_2,\varepsilon_3\in\{\pm1\}}
\left\{\exp\bigl(\mathrm i(\theta_1+\varepsilon_2\theta_2+\varepsilon_3\theta_3)\bigr)-1\right\}.
\]
A Taylor expansion gives
\[
\delta_1(\theta)=\mathrm i\xi_1+O(h|\xi|^2),
\qquad
\delta_2(\theta)=\mathrm i\xi_2+O(h|\xi|^2),
\qquad
\delta_3(\theta)=\mathrm i\xi_3+O(h|\xi|^2).
\]
Thus \(\widehat{\mathcal D}_s=-\mathrm i\widetilde D_s\) has the correct low-frequency principal symbol for \(-\mathrm i\partial_{r_s}\). The numerical \(O(h^2)\) reference slopes shown later are therefore used as empirical benchmarks for the assembled Hermitian stiffness operator and the selected grouped-band metric, rather than as a general a priori convergence theorem for heterogeneous anisotropic media.
\end{remark}

\section{Parity Blocking on Even Grids}
\label{sec:block}

This section proves the parity block structure induced by body-diagonal Bloch differences on even grids. For brevity, write
\[
A_h(\bk):=B_h(\bk)C_hB_h(\bk)^H,
\]
and continue to denote the discrete mass matrix by \(M_h\). The discrete generalized eigenvalue problem is
\[
A_h(\bk)v_h=\omega^2M_hv_h,
\qquad v_h\in\mathbb C^{3n}.
\]

From the phase viewpoint, let
\[
\phi_\ell=2\pi\kappa_\ell=\bq\cdot\ba{\ell},\qquad \ell=1,2,3,
\]
and define the one-dimensional Bloch Fourier phases
\[
\alpha_p=\frac{\phi_1+2\pi p}{n_1},\qquad
\beta_q=\frac{\phi_2+2\pi q}{n_2},\qquad
\gamma_r=\frac{\phi_3+2\pi r}{n_3}.
\]
The frequency-domain phase factors of the four body-diagonal shifts \(S_{d_1},S_{d_2},S_{d_3},S_{d_4}\) are
\begin{equation}\label{eq:Sd_phase_maintext}
e^{\mathrm i(\alpha_p+\beta_q+\gamma_r)},\quad
e^{\mathrm i(\alpha_p+\beta_q-\gamma_r)},\quad
e^{\mathrm i(\alpha_p-\beta_q+\gamma_r)},\quad
e^{\mathrm i(\alpha_p-\beta_q-\gamma_r)}.
\end{equation}
When \(n_1,n_2,n_3\) are all even, a half-period shift in two Fourier indices, for example
\[
(p,q,r)\mapsto\left(p+\frac{n_1}{2},q+\frac{n_2}{2},r\right),
\]
adds \(\pi\) to the corresponding two phases and therefore leaves all four phase factors in \eqref{eq:Sd_phase_maintext} unchanged. Similarly,
\[
(p,q,r)\mapsto\left(p+\frac{n_1}{2},q,r+\frac{n_3}{2}\right),
\qquad
(p,q,r)\mapsto\left(p,q+\frac{n_2}{2},r+\frac{n_3}{2}\right)
\]
also preserve the four body-diagonal phase factors. Thus body-diagonal differences on even grids carry two independent half-period invariances. We next express this structure algebraically on the physical grid.

Let grid points in the unit cell be indexed by integer triples \(x=(i,j,k)\). On the scalar grid function space \(\mathbb C^n\), define
\begin{equation}\label{eq:Gamma12_13_sign_main}
g_{12}(x)=(-1)^{i+j},
\qquad
g_{13}(x)=(-1)^{i+k},
\end{equation}
and the corresponding diagonal operators
\begin{equation}\label{eq:Gamma12_13_op_main}
(\Gamma_{12}w)(x)=g_{12}(x)w(x),
\qquad
(\Gamma_{13}w)(x)=g_{13}(x)w(x).
\end{equation}
They satisfy
\begin{equation}\label{eq:Gamma_basic_main}
\Gamma_{12}^*=\Gamma_{12},\qquad
\Gamma_{13}^*=\Gamma_{13},\qquad
\Gamma_{12}^2=\Gamma_{13}^2=I,
\qquad
\Gamma_{12}\Gamma_{13}=\Gamma_{13}\Gamma_{12}.
\end{equation}
Lift these operators to the velocity space and the Voigt strain/stress space by
\begin{equation}\label{eq:Gamma_lift_main}
\Gamma_a^{(v)}:=I_3\otimes\Gamma_a,
\qquad
\Gamma_a^{(e)}:=I_6\otimes\Gamma_a,
\qquad a\in\{12,13\}.
\end{equation}

\begin{theorem}[Block invariance on even grids]\label{thm:block_invariance}
Assume that \(n_1,n_2,n_3\) are even and that the stiffness matrix \(C_h\) and the mass matrix \(M_h\) are nodewise local multiplication matrices. Then the discrete generalized eigenvalue pair \((A_h(\bk),M_h)\) is compatible with the lifted parity operators:
\begin{equation}\label{eq:A_M_commute_main}
A_h(\bk)\Gamma_a^{(v)}=\Gamma_a^{(v)}A_h(\bk),
\qquad
M_h\Gamma_a^{(v)}=\Gamma_a^{(v)}M_h,
\qquad a\in\{12,13\}.
\end{equation}
Consequently, the velocity space admits an orthogonal decomposition into four common eigenspaces of \((\Gamma_{12}^{(v)},\Gamma_{13}^{(v)})\), and both \(A_h(\bk)\) and \(M_h\) leave each subspace invariant.
\end{theorem}

\begin{proof}
Let \(S_d\) be the Bloch shift with step \(d=(d_1,d_2,d_3)\in\mathbb Z^3\). Since
\[
g_{12}(x+d)=g_{12}(x)(-1)^{d_1+d_2},\qquad
g_{13}(x+d)=g_{13}(x)(-1)^{d_1+d_3},
\]
we have
\begin{equation}\label{eq:Gamma_shift_comm_maintext}
\Gamma_{12}S_d=(-1)^{d_1+d_2}S_d\Gamma_{12},
\qquad
\Gamma_{13}S_d=(-1)^{d_1+d_3}S_d\Gamma_{13}.
\end{equation}
The four body-diagonal shifts used in this paper all have steps \(d=(1,\pm1,\pm1)\), so \(d_1+d_2\) and \(d_1+d_3\) are even. Hence
\[
\Gamma_a S_{d_\ell}=S_{d_\ell}\Gamma_a,
\qquad a\in\{12,13\},\quad \ell=1,2,3,4.
\]
Since \(D_{d_\ell}\) is linear in \(S_{d_\ell}\), and since each reconstructed derivative \(\widetilde D_m\) is a linear combination of \(D_{d_\ell}\), it follows that
\begin{equation}\label{eq:Gamma_Dtilde_comm_main}
\Gamma_a\widetilde D_m=\widetilde D_m\Gamma_a,
\qquad a\in\{12,13\},\quad m=1,2,3.
\end{equation}
The assembly of \(B_h(\bk)\) gives
\begin{equation}\label{eq:B_Gamma_lift_comm_main}
\Gamma_a^{(v)}B_h(\bk)=B_h(\bk)\Gamma_a^{(e)},
\qquad
B_h(\bk)^H\Gamma_a^{(v)}=\Gamma_a^{(e)}B_h(\bk)^H.
\end{equation}
On the other hand, \(C_h\) acts as a local \(6\times6\) Voigt stiffness block at each grid point, and \(M_h\) acts as a local mass block. Neither changes the parity label of a node. Thus
\begin{equation}\label{eq:C_M_Gamma_comm_main}
\Gamma_a^{(e)}C_h=C_h\Gamma_a^{(e)},
\qquad
\Gamma_a^{(v)}M_h=M_h\Gamma_a^{(v)}.
\end{equation}
Combining \eqref{eq:B_Gamma_lift_comm_main} and \eqref{eq:C_M_Gamma_comm_main} yields \eqref{eq:A_M_commute_main}.
\end{proof}

By Theorem~\ref{thm:block_invariance}, for each \((\sigma_{12},\sigma_{13})\in\{\pm1\}^2\), define
\begin{equation}\label{eq:V_sigma_main}
\mathcal V_{\sigma_{12},\sigma_{13}}
=
\{v_h\in\mathbb C^{3n}:
\Gamma_{12}^{(v)}v_h=\sigma_{12}v_h,
\Gamma_{13}^{(v)}v_h=\sigma_{13}v_h\}.
\end{equation}
The corresponding orthogonal projector is
\begin{equation}\label{eq:Proj_sigma_main}
P_{\sigma_{12},\sigma_{13}}
=\frac14
(I+\sigma_{12}\Gamma_{12}^{(v)})
(I+\sigma_{13}\Gamma_{13}^{(v)}).
\end{equation}
Equation~\eqref{eq:Gamma_basic_main} implies that these projectors are mutually orthogonal and sum to the identity. Hence
\begin{equation}\label{eq:direct_sum_block_main}
\mathbb C^{3n}
=
\mathcal V_{+,+}\oplus
\mathcal V_{+,-}\oplus
\mathcal V_{-,+}\oplus
\mathcal V_{-,-}.
\end{equation}

\begin{corollary}[Reduced generalized eigenvalue problem on each block]\label{cor:block_reduced_gep}
Let \(Q_\sigma\) be an orthonormal basis matrix for \(\mathcal V_\sigma\), and define
\begin{equation}\label{eq:Asigma_main}
A_{\sigma,h}(\bk):=Q_\sigma^*A_h(\bk)Q_\sigma,
\qquad
M_{\sigma,h}:=Q_\sigma^*M_hQ_\sigma.
\end{equation}
Then the original problem restricted to the block \(\mathcal V_\sigma\) is equivalent to
\begin{equation}\label{eq:reduced_eig_main}
A_{\sigma,h}(\bk)y=\omega^2M_{\sigma,h}y,
\qquad v_{\sigma,h}=Q_\sigma y.
\end{equation}
\end{corollary}
The parity blocking therefore provides an exact algebraic block decomposition: the original generalized eigenvalue problem can be solved on four invariant subspaces independently. The decomposition does not require the material parameters to be homogeneous within the unit cell, but it does require \(C_h\) and \(M_h\) to enter as nodewise local multiplication matrices. If the material discretization contains nonlocal cross-node averaging, filtering, or higher-order interface coupling terms, the commutation relations above must be reexamined. Whether the four blocks have additional spectral equivalence or unitary similarity depends on the material distribution and geometric symmetry, and does not follow from the block structure alone.

\section{Fourier SVD Block Reduction}
\label{sec:block_wsvd}

The full-space discrete operator
\[
A_h(\bk)=B_h(\bk)C_hB_h(\bk)^H
\]
and the even grid invariant block decomposition have been established above. Fourier SVD variable transformations have been used for fast solvers of 3D phononic crystal linear elastic eigenvalue problems. Their main idea is to exploit the Fourier structure of Bloch periodic difference operators to separate the derivative part and to use local SVDs to improve the inner Krylov linear solves \cite{LyuTianLiLin2024JSC}.

The Fourier SVD structure acts primarily on the derivative operators. If the stiffness matrix \(C_h\) corresponds to a spatially heterogeneous material, then \(C_h\) is generally not diagonal by mode in either full-space Fourier coordinates or block Fourier coordinates; modal coupling remains. Therefore, the formulation below does not imply a complete decomposition of the heterogeneous material problem into independent single-mode problems. Instead, it structures the Bloch periodic derivative operator and provides a unified variable transformation and preconditioning framework for the shift-invert linear subproblems.

\subsection{Fourier Coordinates in the Full Space}

Let
\[
n=n_1n_2n_3,
\qquad
\mathbb Z_m:=\{0,1,\ldots,m-1\},
\qquad
\mathcal I:=\mathbb Z_{n_1}\times\mathbb Z_{n_2}\times\mathbb Z_{n_3}.
\]
Here \(\mathcal I\) denotes the full space Fourier frequency index set. We use \(I_m\) for the \(m\times m\) identity matrix and write \(I\) when the dimension is clear from context. By Section~\ref{sec:blochdisc}, the three Bloch periodic shift matrices \(H_1,H_2,H_3\) commute pairwise and can therefore be simultaneously diagonalized by a common 3D Bloch Fourier basis. Let
\begin{equation}\label{eq:T_full_def_unified}
T=X_3\otimes X_2\otimes X_1\in\mathbb C^{n\times n},
\end{equation}
where \(X_\ell\) is the unitary eigenvector matrix of the one-dimensional Bloch shift in the \(\ell\)-th direction. Then all body-diagonal shifts, body-diagonal differences, and reconstructed derivative operators built from \(H_1,H_2,H_3\) admit frequency mode representations under \(T\).

For the full-space formulation, the 3D Bloch DFT basis functions are
\begin{equation}\label{eq:psi_pqr_unified}
\psi_{pqr}(i,j,k)
=\frac1{\sqrt n}\exp\{\mathrm i(i\alpha_p+j\beta_q+k\gamma_r)\},
\end{equation}
where
\[
\alpha_p=\frac{\phi_1+2\pi p}{n_1},\qquad
\beta_q=\frac{\phi_2+2\pi q}{n_2},\qquad
\gamma_r=\frac{\phi_3+2\pi r}{n_3}.
\]
Here \((p,q,r)\in\mathcal I\). This complete frequency set is the parent space for the quotient-frequency sets used in the block representation.

\subsection{Fourier Coordinates in the Block Spaces}

Assume now that \(n_1,n_2,n_3\) are even. The parity multipliers in Section~\ref{sec:block} correspond to half-period shifts in the full-space Fourier indices. Specifically,
\[
(-1)^{i+j}\psi_{pqr}=\psi_{p+n_1/2,\,q+n_2/2,\,r},
\qquad
(-1)^{i+k}\psi_{pqr}=\psi_{p+n_1/2,\,q,\,r+n_3/2}.
\]
Define
\[
h_{12}=\left(\frac{n_1}{2},\frac{n_2}{2},0\right),\qquad
h_{13}=\left(\frac{n_1}{2},0,\frac{n_3}{2}\right),\qquad
h_{23}=h_{12}+h_{13}.
\]
For \(m\in\mathcal I\), the orbit generated by these half-period shifts is
\begin{equation}\label{eq:orbit_unified}
\mathcal O(m)=\{m,\ m+h_{12},\ m+h_{13},\ m+h_{23}\}.
\end{equation}
Each orbit contains four frequency points. Choose a representative frequency set \(\mathcal R\subset\mathcal I\) such that each orbit has exactly one representative. Then \(|\mathcal R|=n/4\).

For a fixed block sign \(\sigma=(\sigma_{12},\sigma_{13})\in\{\pm1\}^2\) and \(m\in\mathcal R\), define the folding vector
\begin{equation}\label{eq:f_sigma_unified}
f_m^{(\sigma)}
=\frac12\left(e_m+\sigma_{12}e_{m+h_{12}}
+\sigma_{13}e_{m+h_{13}}
+\sigma_{12}\sigma_{13}e_{m+h_{23}}\right).
\end{equation}
Let
\[
F_\sigma=(f_m^{(\sigma)})_{m\in\mathcal R}\in\mathbb C^{n\times(n/4)}.
\]
Since different orbits are disjoint and since the normalization in \eqref{eq:f_sigma_unified} is \(1/2\),
\[
F_\sigma^HF_\sigma=I_{n/4}.
\]
The Fourier basis for the \(\sigma\)-th block is
\begin{equation}\label{eq:Tsigma_unified}
T_\sigma:=TF_\sigma\in\mathbb C^{n\times(n/4)},
\qquad
T_\sigma^HT_\sigma=I_{n/4}.
\end{equation}
Thus the essential distinction between the full-space and a block space is not the SVD formula itself, but the reduction of Fourier degrees of freedom from the complete frequency set \(\mathcal I\) to the quotient representative set \(\mathcal R\).

\subsection{Symbol Blocks and Material Coupling}

To avoid separate derivations for the full and block spaces, introduce the unified index
\[
\xi\in \{\mathrm{full}\}\cup\{\sigma:\sigma=(\sigma_{12},\sigma_{13})\in\{\pm1\}^2\},
\]
where \(\xi=\mathrm{full}\) denotes the full space and \(\xi=\sigma\) denotes one of the four blocks. Define
\begin{equation}\label{eq:T_xi_I_xi_def}
T_\xi=
\begin{cases}
T, & \xi=\mathrm{full},\\
T_\sigma, & \xi=\sigma,
\end{cases}
\qquad
\mathcal I_\xi=
\begin{cases}
\mathcal I, & \xi=\mathrm{full},\\
\mathcal R, & \xi=\sigma.
\end{cases}
\end{equation}
Here \(\mathcal I_\xi\) is the Fourier frequency index set associated with the corresponding space. In this notation, the divergence type matrix associated with the derivative operator is written in the corresponding Fourier coordinates as
\begin{equation}\label{eq:Bhat_xi_def}
\widehat B_{\xi,h}
:=(I_3\otimes T_\xi)^HB_h(\bk)(I_6\otimes T_\xi)
=\bigoplus_{m\in\mathcal I_\xi}\widehat B_{\xi,h}(m).
\end{equation}
Let \(s_{d_\ell,\xi}(m)\) be the symbol of the body-diagonal shift \(S_{d_\ell}\) in the \(T_\xi\) coordinates, and define
\[
d_{\ell,\xi}(m):=\frac{s_{d_\ell,\xi}(m)-1}{h_d},
\qquad \ell=1,2,3,4.
\]
The symbols of the three reconstructed derivatives are
\begin{align}
\delta_{1,\xi}(m)&=c_1\{d_{1,\xi}(m)+d_{2,\xi}(m)+d_{3,\xi}(m)+d_{4,\xi}(m)\},\label{eq:delta1_xi_unified}\\
\delta_{2,\xi}(m)&=c_2\{d_{1,\xi}(m)+d_{2,\xi}(m)-d_{3,\xi}(m)-d_{4,\xi}(m)\},\label{eq:delta2_xi_unified}\\
\delta_{3,\xi}(m)&=c_3\{d_{1,\xi}(m)-d_{2,\xi}(m)+d_{3,\xi}(m)-d_{4,\xi}(m)\}.\label{eq:delta3_xi_unified}
\end{align}
For each frequency mode \(m\in\mathcal I_\xi\), the local symbol block is
\begin{equation}\label{eq:Bhat_xi_m_def}
\widehat B_{\xi,h}(m)=
\begin{bmatrix}
\eta_{1,\xi}(m)&0&0&0&\eta_{3,\xi}(m)&\eta_{2,\xi}(m)\\
0&\eta_{2,\xi}(m)&0&\eta_{3,\xi}(m)&0&\eta_{1,\xi}(m)\\
0&0&\eta_{3,\xi}(m)&\eta_{2,\xi}(m)&\eta_{1,\xi}(m)&0
\end{bmatrix},
\end{equation}
where \(\eta_{j,\xi}(m)=-\mathrm i\,\delta_{j,\xi}(m)\). For \(\xi=\mathrm{full}\), \(m\) ranges over the full frequency set \(\mathcal I\). For \(\xi=\sigma\), \(m\) ranges only over the quotient representative set \(\mathcal R\), because the four frequency points in the same orbit have identical body-diagonal shift symbols.

The material matrix in the corresponding Fourier coordinates is
\begin{equation}\label{eq:CF_xi_def}
C_{\xi,h}^{(F)}:=(I_6\otimes T_\xi)^HC_h(I_6\otimes T_\xi).
\end{equation}
The reduced velocity side operator is therefore
\begin{equation}\label{eq:Av_xi_def}
A_{\xi,h}^{(v)}=\widehat B_{\xi,h} C_{\xi,h}^{(F)}\widehat B_{\xi,h}^H.
\end{equation}
Even though \(\widehat B_{\xi,h}\) is block diagonal by frequency modes or quotient-frequency modes, \(C_{\xi,h}^{(F)}\) is generally not diagonal in \(m\) for spatially heterogeneous materials. Hence \(A_{\xi,h}^{(v)}\) usually still contains intermode couplings. A stronger modewise decoupling occurs only when \(C_h\) is spatially constant, or when a constant reference stiffness is used in a weighted preconditioner.

\subsection{SVD Reduction with Optional Weighting}

Let \(\Psi\in\mathbb C^{6\times6}\) be a constant Hermitian positive definite matrix. The following formulas allow a general weight matrix in order to describe possible reference material preconditioners. Unless otherwise stated, the numerical experiments use \(\Psi=I_6\), so the implementation uses the unweighted local SVD. If a constant reference material weighting is used in a timing run, the corresponding \(\Psi\) is reported with the solver settings. For any \(\xi\in \{\mathrm{full}\}\cup\{\sigma:\sigma\in\{\pm1\}^2\}\) and \(m\in\mathcal I_\xi\), take the ordinary SVD of the weighted local symbol block:
\begin{equation}\label{eq:weighted_svd_unified}
\widehat B_{\xi,h}(m)\Psi^{-1/2}
=P_{\xi,m}\Sigma_{\xi,m}\widetilde Q_{\xi,m}^H.
\end{equation}
Set
\[
Q_{\xi,m}:=\Psi^{-1/2}\widetilde Q_{\xi,m},
\qquad
Q_{\xi,m}^{H,\Psi}:=Q_{\xi,m}^H\Psi .
\]
Then
\begin{equation}\label{eq:weighted_svd_local_unified}
\widehat B_{\xi,h}(m)=P_{\xi,m}\Sigma_{\xi,m}Q_{\xi,m}^{H,\Psi},
\qquad
Q_{\xi,m}^{H,\Psi}Q_{\xi,m}=I.
\end{equation}
The superscript \(H,\Psi\) denotes the \(\Psi\)-weighted Hermitian adjoint. When \(\Psi=I_6\), \(Q_{\xi,m}^{H,\Psi}=Q_{\xi,m}^H\), and the formula reduces to the standard local SVD. This is the setting used in all numerical results. The general \(\Psi\)-notation is retained only to show that the variable transformation naturally extends to reference material weighting and to avoid confusing ordinary Hermitian transposes with weighted adjoints.

Concatenate the local factors over \(m\in\mathcal I_\xi\), and define
\[
\Psi_\xi:=I_{|\mathcal I_\xi|}\otimes\Psi,
\qquad
Q_\xi^{H,\Psi}:=Q_\xi^H\Psi_\xi,
\]
where \(I_{|\mathcal I_\xi|}\) is the identity matrix of order \(|\mathcal I_\xi|\). Then
\begin{equation}\label{eq:Bhat_xi_wsvd}
\widehat B_{\xi,h}=P_\xi\Sigma_\xi Q_\xi^{H,\Psi}.
\end{equation}
Consequently,
\begin{equation}\label{eq:Av_xi_factor}
A_{\xi,h}^{(v)}
=P_\xi\Sigma_\xi
\left(Q_\xi^{H,\Psi}C_{\xi,h}^{(F)}(Q_\xi^{H,\Psi})^H\right)
\Sigma_\xi^H P_\xi^H.
\end{equation}
Define the material core matrix
\begin{equation}\label{eq:A_xi_core_def}
A_{\xi,h,e}:=Q_\xi^{H,\Psi}C_{\xi,h}^{(F)}(Q_\xi^{H,\Psi})^H.
\end{equation}
For the linear system
\[
A_{\xi,h}^{(v)}w_\xi^{(v)}=b_\xi^{(v)},
\]
introduce the two sided variables
\begin{equation}\label{eq:two_sided_variables_unified}
w_{\xi,e}:=\Sigma_\xi^H P_\xi^H w_\xi^{(v)},
\qquad
b_{\xi,e}:=\Sigma_\xi^\dagger P_\xi^H b_\xi^{(v)}.
\end{equation}
On the nonzero singular value subspace, we obtain the core system
\begin{equation}\label{eq:core_system_unified}
A_{\xi,h,e}w_{\xi,e}=b_{\xi,e}.
\end{equation}
Here \(\Sigma_\xi^\dagger\) denotes the Moore Penrose inverse. The use of \(\Sigma_\xi^H\) in the definition of \(w_{\xi,e}\) keeps the transformation dimensionally consistent for both compact and rectangular SVD notation; in the compact diagonal case it reduces to the usual diagonal scaling. At the \(\Gamma\) point or when discrete zero-derivative modes are present, this transformation must be combined with the nullspace projection or deflation treatment in Section~\ref{subsec:gamma_deflation}.

For \(\xi=\mathrm{full}\), \eqref{eq:Bhat_xi_wsvd} and \eqref{eq:core_system_unified} give the Fourier SVD variable transformation in the full space. For \(\xi=\sigma\), they give the corresponding transformation in the quotient-frequency block space. Thus the full and block spaces use the same derivative operator normalization principle; their differences lie mainly in the frequency sets, basis dimensions, and the number of block subproblems that can be solved in parallel.

\section{Discrete Spectral Solver}
\label{sec:impl}

This section describes how the Fourier SVD representation is used in band structure computations, with emphasis on the inverse Lanczos iteration, block merging, matrix-vector products, and \(\Gamma\)-point zero mode treatment. All implementations are based on the discrete operator
\[
A_h(\bk)=B_h(\bk)C_hB_h(\bk)^H
\]
from Section~\ref{sec:blochdisc}. The only difference between the full and block spaces is whether this operator is restricted to the invariant subspaces described in Section~\ref{sec:block}.

\subsection{Lanczos Iteration and Matrix Vector Products}

For a given Bloch wave vector \(\bk\), the discrete generalized eigenvalue problem is
\begin{equation}\label{eq:impl_gep_unified}
A_h(\bk)v_h=\omega^2M_hv_h,
\qquad
A_h(\bk)=B_h(\bk)C_hB_h(\bk)^H.
\end{equation}
After 3D grid-refinement, the number of velocity degrees of freedom is \(3n=3n_1n_2n_3\), making a full eigenvalue decomposition infeasible. Since only a finite number of low-frequency bands are needed, we use an outer Krylov--Schur or Lanczos type iteration.

In the shift-invert framework, one operator application in the outer iteration is converted into a linear solve. Given a positive shift \(\delta>0\), a typical inner system is
\begin{equation}\label{eq:shift_invert_system_unified}
\bigl(A_h(\bk)+\delta M_h\bigr)x=M_hq.
\end{equation}
Here \(q\) is the trial vector provided by the outer Krylov--Schur or Lanczos iteration, and \(x\) is the vector returned by this shift-invert application. This linear system is solved approximately by PCG or an appropriate Hermitian Krylov method. The Fourier SVD variable transformation normalizes the derivative part and provides the framework for the inner preconditioner. At the \(\Gamma\) point or in the presence of discrete zero-derivative modes, it is combined with the deflation or nullspace postprocessing described in Section~\ref{subsec:gamma_deflation}.

For the \(\sigma\)-th block, the generalized eigenvalue problem is
\begin{equation}\label{eq:block_gep_impl_unified}
A_{\sigma,h}(\bk)y=\omega^2M_{\sigma,h}y,
\qquad
A_{\sigma,h}(\bk)=Q_\sigma^*A_h(\bk)Q_\sigma,
\qquad
M_{\sigma,h}=Q_\sigma^*M_hQ_\sigma.
\end{equation}
After \(y\) is computed, the eigenvector in the original velocity space is recovered by \(v_{\sigma,h}=Q_\sigma y\). The eigenvalues obtained from the four blocks are merged and sorted to produce the spectrum of the corresponding full-space discrete problem.

The full and block spaces share the same nodewise material multiplication, derivative symbols, shift-invert outer spectral iteration, and PCG inner solve. Their main differences are the number of degrees of freedom, the Fourier basis, the frequency index set, and the spectral merging procedure. The full space uses the full Fourier basis \(T\) and the complete frequency set \(\mathcal I\), whereas the \(\sigma\)-th block uses \(T_\sigma=TF_\sigma\) and the quotient-frequency set \(\mathcal R\). Thus the block acceleration comes from dimension reduction and parallel organization of the four subproblems, not from changing the underlying discrete spectral problem.

In the matrix-vector implementation, the stiffness operator is never formed explicitly as a full matrix. Instead, it is applied as
\[
x\mapsto B_h(\bk)^Hx
\mapsto C_hB_h(\bk)^Hx
\mapsto B_h(\bk)C_hB_h(\bk)^Hx.
\]
Since \(C_h\) and \(M_h\) are nodewise local multiplication matrices, material multiplication, Bloch periodic differences, FFT type transforms, and vector updates in the PCG iteration are all suitable for GPU execution. The GPU version accelerates matrix-vector products and inner linear solves without changing the discrete spectral problem. The block version further restricts the same operator to the four invariant subspaces.

\subsection{\texorpdfstring{\(\Gamma\)-Point Deflation}{Gamma Point Deflation}}
\label{subsec:gamma_deflation}

At the \(\Gamma\) point, all Bloch phases equal one. The body-diagonal difference operators then have a nontrivial zero-derivative space. In addition to the three physical rigid-translation modes, the even grid body-diagonal construction preserves three checkerboard-type scalar modes. As shown in Appendix~\ref{app:gamma_null}, the full velocity space contains \(12\) zero-derivative modes,
\begin{equation}\label{eq:gamma_null_space_main}
\mathcal N_\Gamma
=
\operatorname{span}\{\chi_\mu e_\alpha:
\mu\in\{0,12,13,23\},\ \alpha=1,2,3\},
\end{equation}
where
\[
\begin{aligned}
\chi_0(i,j,k)&=1,
&\chi_{12}(i,j,k)&=(-1)^{i+j},\\
\chi_{13}(i,j,k)&=(-1)^{i+k},
&\chi_{23}(i,j,k)&=(-1)^{j+k}.
\end{aligned}
\]
Only \(\chi_0 e_\alpha\) correspond to physical rigid translations; the other modes are discrete checkerboard modes induced by the body-diagonal differences. Hence the \(\Gamma\)-point treatment must account for all \(12\) zero-derivative modes rather than only the three physical modes.

Let \(Q_0\in\mathbb C^{3n\times12}\) be an \(M_h\)-orthonormal basis of \(\mathcal N_\Gamma\), i.e., \(Q_0^HM_hQ_0=I\). The \(\Gamma\)-point stiffness matrix is modified by
\begin{equation}\label{eq:gamma_deflation_full_main}
A_{\Gamma,\mathrm{def}}
=
A_h(\Gamma)+d_0M_hQ_0Q_0^HM_h,
\qquad d_0>0.
\end{equation}
This moves the generalized eigenvalues in the nullspace directions to \(d_0\) while leaving the nonzero spectral structure on the \(M_h\)-orthogonal complement unchanged. In postprocessing, the corresponding zero-derivative modes are skipped or explicitly removed from relative error statistics.

When the problem is solved block by block, each block contains three zero-derivative modes at the \(\Gamma\) point. In that case, an \(M_{\sigma,h}\)-orthonormal basis \(Q_{0,\sigma}\) is constructed in each block, and
\begin{equation}\label{eq:gamma_deflation_sector_main}
A_{\sigma,\Gamma,\mathrm{def}}
=
A_{\sigma,h}(\Gamma)+d_0M_{\sigma,h}Q_{0,\sigma}Q_{0,\sigma}^HM_{\sigma,h},
\qquad
Q_{0,\sigma}^HM_{\sigma,h}Q_{0,\sigma}=I.
\end{equation}
After this treatment, the full and block spaces can use the same shift-invert Lanczos and inner PCG framework at the \(\Gamma\) point as at general \(\bk\) points.

\section{Numerical Experiments}
\label{sec:num}

This section validates the proposed discrete structure using a 3D two-phase anisotropic phononic crystal example. The numerical tests are organized to separate external frequency accuracy, algebraic consistency between the full and block implementations, empirical grid-refinement behavior, and computational cost. We first specify the material parameters, subpixel volume-fraction assignment rule, computing environment, and solver settings. We then show the band structure obtained by the present implementation and quantify the spectral agreement between the full-space and block implementations on a \(60\times60\times60\) grid with \(N_k=91\) Bloch-path sampling points and the first \(N_b=40\) bands. A comparison with an independent COMSOL model with \(342459\) degrees of freedom is used to assess the frequency discrepancy of the grouped \(128^3\) computed spectrum. This degree-of-freedom count refers to the finite element reference model and is reported to specify the external comparison scale. Using the full-space GPU spectrum as an internal reference for the same discrete model, we report bandwise and per-wave-vector errors of the merged block GPU spectrum. We further use the \(256^3\) full-space GPU result as a high-resolution reference to present empirical grid-refinement curves at selected high-symmetry Bloch wave vectors. Finally, we report a grid-refinement iteration and runtime diagnostic, followed by the end-to-end wall-clock times of the full-space CPU, full-space GPU, and block GPU implementations.

Unless otherwise specified, all experiments use the Bloch periodic rotated staggered derivative reconstruction described in Section~\ref{sec:blochdisc} and solve
\begin{equation}
\label{eq:num_full_problem}
A_h(\bk)v_h=\omega^2 M_hv_h,
\qquad
A_h(\bk)=B_h(\bk)C_hB_h(\bk)^H .
\end{equation}
The frequencies in the band diagrams are defined by \(f=\omega/(2\pi)\) and are reported in Hz. The Bloch path is
\[
\Gamma\to X\to W\to K\to\Gamma\to L\to U\to W\to L\to K .
\]
Each segment contains \(10\) sampling intervals, giving \(N_k=91\) wave vector samples on the full path. In reduced reciprocal coordinates, defined by \(\bq\cdot\ba{\ell}=2\pi\kappa_\ell\), the high-symmetry points used in the computation are
\begin{equation}
\label{eq:fcc_high_symmetry_points}
\begin{aligned}
\Gamma&=(0,0,0),
&X&=\left(0,\frac12,\frac12\right),
&W&=\left(\frac14,\frac12,\frac34\right),\\
K&=\left(\frac38,\frac38,\frac34\right),
&L&=\left(\frac12,\frac12,\frac12\right),
&U&=\left(\frac14,\frac58,\frac58\right).
\end{aligned}
\end{equation}

\subsection{Materials and Solver Settings}
\label{subsec:num_setting}

The computational domain is an FCC periodic cell. The lattice constant is
\(a_0=3.0\times10^{-2}\,\mathrm{m}\), and the lattice basis vectors are
\begin{equation}
\label{eq:num_lattice_vectors}
\bm a_1=(0,a_0/2,a_0/2)^\top,\qquad
\bm a_2=(a_0/2,0,a_0/2)^\top,\qquad
\bm a_3=(a_0/2,a_0/2,0)^\top .
\end{equation}
The grid points are indexed in the reduced coordinate \(\br\in[0,1)^3\), while the elastic stiffness components are interpreted in the Cartesian physical frame. Accordingly, the derivative operators used in \(B_h(\bk)\) are the Cartesian combinations obtained from \(\nabla_{\mathbf x}=A_{\rm lat}^{-T}\nabla_{\br}\), as described in \eqref{eq:gradient_transform_main} and \eqref{eq:Dcal_cartesian_transform_main}. The reduced Bloch vector \(\bk\) on the high-symmetry path determines the boundary phase by \(\phi_\ell=2\pi\kappa_\ell\).

The material consists of a low quartz matrix and a spherical silicon inclusion. The room temperature elastic constants of low quartz are taken from experimentally averaged values reported in the literature, and the density and cubic elastic constants of silicon are taken from standard semiconductor material data \cite{PabstGregorova2013,IOFFESilicon}. The spherical inclusion is centered at the reduced coordinate \((1/2,1/2,1/2)\) with radius \(R_{\rm sph}/a_0=0.25\). We use subpixel smoothing, equivalently a volume fraction or cut-cell averaging rule, to assign material parameters. Cells entirely inside one material phase are assigned the parameters of that phase. For cells cut by the spherical interface, the local silicon volume fraction \(\theta_i\in[0,1]\) is computed and the local density and stiffness are assigned by
\begin{equation}
\label{eq:subpixel_material_average}
\rho_i=(1-\theta_i)\rho_1+\theta_i\rho_2,
\qquad
C_i=(1-\theta_i)C^{(1)}+\theta_i C^{(2)} .
\end{equation}
This treatment reduces the geometric error caused by voxelization of the curved interface, especially in the grid-refinement tests. The linear volume fraction average is used only as a numerical device to mitigate curved interface geometry error; it is not intended as a rigorous local effective medium model. The full-space, block-space, refinement, and timing comparisons all use the same material assignment rule.

We use the Voigt ordering \([11,22,33,23,13,12]\) and the engineering shear convention. The density of low quartz is
\(\rho_1=2650\,\mathrm{kg\,m^{-3}}\), and its nonzero independent stiffness entries are
\begin{equation}
\label{eq:num_material_C1_nonzero}
\begin{aligned}
&C^{(1)}_{11}=C^{(1)}_{22}=86.6,
\qquad C^{(1)}_{12}=6.9,
\qquad C^{(1)}_{13}=C^{(1)}_{23}=13.6,\\
&C^{(1)}_{33}=106.7,
\qquad C^{(1)}_{14}=-17.7,
\qquad C^{(1)}_{24}=17.7,\\
&C^{(1)}_{44}=C^{(1)}_{55}=57.7,
\qquad C^{(1)}_{56}=-17.7,
\qquad C^{(1)}_{66}=39.85 .
\end{aligned}
\end{equation}
The density of the silicon inclusion is \(\rho_2=2329\,\mathrm{kg\,m^{-3}}\), and its nonzero independent stiffness entries are
\begin{equation}
\label{eq:num_material_C2_nonzero}
\begin{aligned}
&C^{(2)}_{11}=C^{(2)}_{22}=C^{(2)}_{33}=166.0,
\qquad C^{(2)}_{12}=C^{(2)}_{13}=C^{(2)}_{23}=64.0,\\
&C^{(2)}_{44}=C^{(2)}_{55}=C^{(2)}_{66}=79.6 .
\end{aligned}
\end{equation}
All stiffness entries above are listed in \(\mathrm{GPa}\). In the assembled generalized eigenvalue problem, these entries are converted to SI units by multiplying by \(10^9\), so that the reported frequencies are in Hz after using \(f=\omega/(2\pi)\). Unlisted entries are determined by matrix symmetry or are zero. The signs of the \(C_{14}\)-related terms in low quartz depend on the handedness of the crystal coordinates and the Voigt ordering convention. We adopt the coordinate convention shown in \eqref{eq:num_material_C1_nonzero} and use it consistently in all computations.

All \(60\times60\times60\) grid experiments compute the first \(N_b=40\) bands. The full-space velocity degrees of freedom equal \(3\times60^3=648000\). The Bloch path uses the \(N_k=91\) samples specified at the beginning of Section~\ref{sec:num}.

\paragraph{Computing environment and solver parameters.}
All programs are run in MATLAB R2024b. The CPU version is tested on a CPU queue, where each node has \(128\) CPU cores and \(384\,\mathrm{GB}\) of memory. The GPU version is tested on the \texttt{gpu\_v100} queue, whose nodes have \(512\,\mathrm{GB}\) memory and multiple NVIDIA Tesla V100 GPUs, each with \(32\,\mathrm{GB}\) memory. The full-space GPU test uses one V100 GPU. The block GPU test uses a MATLAB parallel pool to assign the four blocks to four V100 GPUs.

The outer eigenvalue problem is solved by MATLAB's \texttt{eigs}, and the inner shift-invert linear systems are solved by PCG. Unless otherwise specified, both the outer \texttt{eigs} tolerance and the inner PCG tolerance are \(10^{-12}\). The maximum number of outer iterations is \(300\), and the maximum number of inner PCG iterations is \(200\). The positive shift \(\delta\) is set internally after nondimensionalization of the assembled generalized eigenproblem. In the present implementation we use
\begin{equation}
\label{eq:num_shift_delta_rule}
\delta=0.30\,\lambda_{\rm scale},
\qquad
\lambda_{\rm scale}:=C_{\max},
\qquad
C_{\max}:=\max_{\ell=1,2}\max_{1\le p,q\le6}
\bigl|C^{(\ell)}_{pq}\bigr|,
\end{equation}
where \(\lambda_{\rm scale}\) is used only as a fixed numerical scale for the shifted linear systems. The same value of \(\delta\) is used in the corresponding full-space and block experiments, so the shift does not affect the full block comparison. At the \(\Gamma\) point, the deflation parameter is \(d_0=\delta\). In the \(60^3\)-grid, \(N_k=91\) tests, the average number of inner PCG iterations is about \(17\). The number of outer \texttt{eigs} iterations is not recorded as a separate diagnostic; thus, timing is reported as the wall-clock time of one complete program run, without separately decomposing the contributions from the outer spectral iteration, inner linear solves, preconditioner construction, or data-transfer.

\subsection{Band Structures}
\label{subsec:num_full_block_bands}

Figure~\ref{fig:num_full_block_bands} shows the band structure obtained by the present implementation under the material setting, Bloch path, number of bands, and solver parameters described in Section~\ref{subsec:num_setting}. The full-space and block spectral agreement is quantified separately in Section~\ref{subsec:num_full_block_error}.

\begin{figure}[htbp]
\centering
\includegraphics[width=0.72\textwidth]{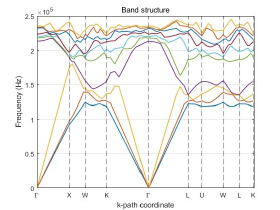}
\caption{Band structure obtained by the present implementation on a \(60\times60\times60\) grid. The result is obtained under the material setting, Bloch path, and solver parameters described in Section~\ref{subsec:num_setting}.}
\label{fig:num_full_block_bands}
\end{figure}

For the present symmetric benchmark, the full-space computation on sufficiently fine grids produces eigenvalues that appear in near-fourfold clusters. The parity decomposition proved in Section~\ref{sec:block} implies that the full-space spectrum is the union of the four block spectra, but it does not imply fourfold degeneracy by itself. The observed near clustering is therefore regarded as a feature of the chosen geometry and material setting rather than as a general consequence of parity blocking. The four values in each local cluster are not interpreted as four well separated physical branches; their within-cluster spread is reported as a small numerical splitting. Figure~\ref{fig:near_fourfold_cluster} illustrates this local structure. The COMSOL comparison in Section~\ref{subsec:num_comsol_reference} provides an external discrepancy scale for interpreting the cluster spread, and the grouped diagnostics in Section~\ref{subsec:num_band_convergence} should be read in this benchmark-specific sense.

\begin{figure}[htbp]
\centering
\includegraphics[width=0.62\textwidth]{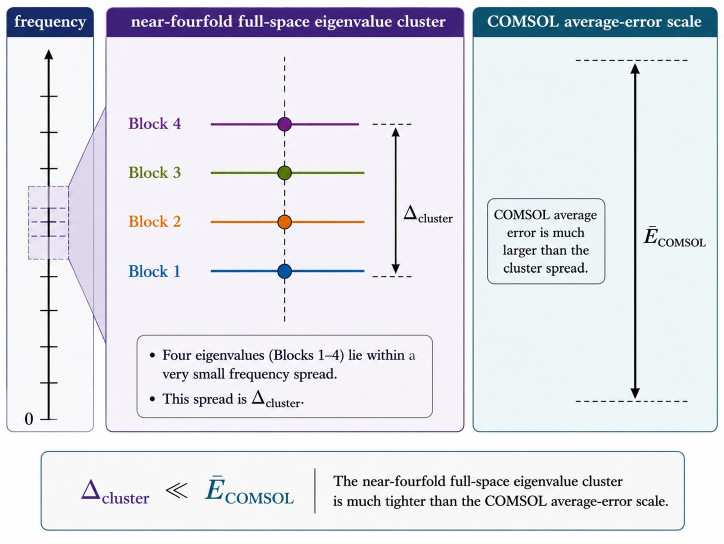}
\caption{Schematic illustration of a near-fourfold eigenvalue cluster in the full-space computation. The block decomposition guarantees that the full-space spectrum is the union of the four block spectra, but it does not by itself imply fourfold degeneracy. The cluster spread is therefore reported as a numerical diagnostic for the present symmetric benchmark.}
\label{fig:near_fourfold_cluster}
\end{figure}

\subsection{Comparison with COMSOL Reference}
\label{subsec:num_comsol_reference}

This subsection reports the external reference comparison between the COMSOL model and the result obtained by the present method. The COMSOL model uses \(342459\) finite element degrees of freedom. This number is stated here because the COMSOL result is used as an independent external frequency reference rather than as an algebraic counterpart of the staggered-grid system. The result from the present method is taken from the full-space \(128^3\) computation and contains the first \(40\) bands along the same \(N_k=91\) Bloch path. To match the ten COMSOL bands, the computed bands are grouped by averaging every four consecutive bands. The comparison therefore uses ten grouped bands, denoted by \(G1,\ldots,G10\). Zero-frequency modes at the \(\Gamma\) point are excluded from relative error statistics.

Over all compared nonzero entries, the maximum absolute discrepancy is \(2.304220\times10^3\,\mathrm{Hz}\), occurring at \(i_k=34\) in group \(G10\), and the maximum relative discrepancy is \(9.705655\times10^{-3}\), also occurring at \(i_k=34\) in group \(G10\). The mean absolute discrepancy over all entries is \(1.720223\times10^2\,\mathrm{Hz}\), and the mean relative discrepancy over valid nonzero entries is \(8.775761\times10^{-4}\). Table~\ref{tab:comsol_code_comparison} reports the high-symmetry point comparison. For each high-symmetry sample, the table lists the band group with the largest relative discrepancy. Repeated high-symmetry labels are distinguished by their occurrence along the path.

\begin{table}[htbp]
\centering
\caption{Comparison between the COMSOL reference model with \(342459\) finite element degrees of freedom and the full-space \(128^3\) result obtained by the present method. The present implementation uses four-band averaging to match the ten COMSOL bands. For each high-symmetry sample, the row reports the band group with the largest relative discrepancy. Frequencies and absolute errors are reported in Hz.}
\label{tab:comsol_code_comparison}
\small
\renewcommand{\arraystretch}{1.12}
\begin{tabular}{cccccc}
\toprule
Wave vector & Group & COMSOL & Present result & Absolute error & Relative error \\
\midrule
\(\Gamma\) & \(G7\)  & \(2.2356\times10^5\) & \(2.2241\times10^5\) & \(1.1536\times10^3\) & \(5.1603\times10^{-3}\) \\
\(X\)        & \(G8\)  & \(2.2324\times10^5\) & \(2.2297\times10^5\) & \(2.6945\times10^2\) & \(1.2070\times10^{-3}\) \\
\(W\)      & \(G1\)  & \(1.2455\times10^5\) & \(1.2441\times10^5\) & \(1.4041\times10^2\) & \(1.1274\times10^{-3}\) \\
\(K\)      & \(G2\)  & \(1.2640\times10^5\) & \(1.2628\times10^5\) & \(1.1558\times10^2\) & \(9.1438\times10^{-4}\) \\
\(L\)      & \(G4\)  & \(1.3537\times10^5\) & \(1.3522\times10^5\) & \(1.5375\times10^2\) & \(1.1358\times10^{-3}\) \\
\(U\)        & \(G2\)  & \(1.2640\times10^5\) & \(1.2628\times10^5\) & \(1.1558\times10^2\) & \(9.1438\times10^{-4}\) \\
\bottomrule
\end{tabular}
\end{table}

Figure~\ref{fig:code_vs_comsol_error_panels} visualizes the same comparison by band group and along the Bloch path.

\begin{figure}[htbp]
\centering
\includegraphics[width=0.96\textwidth]{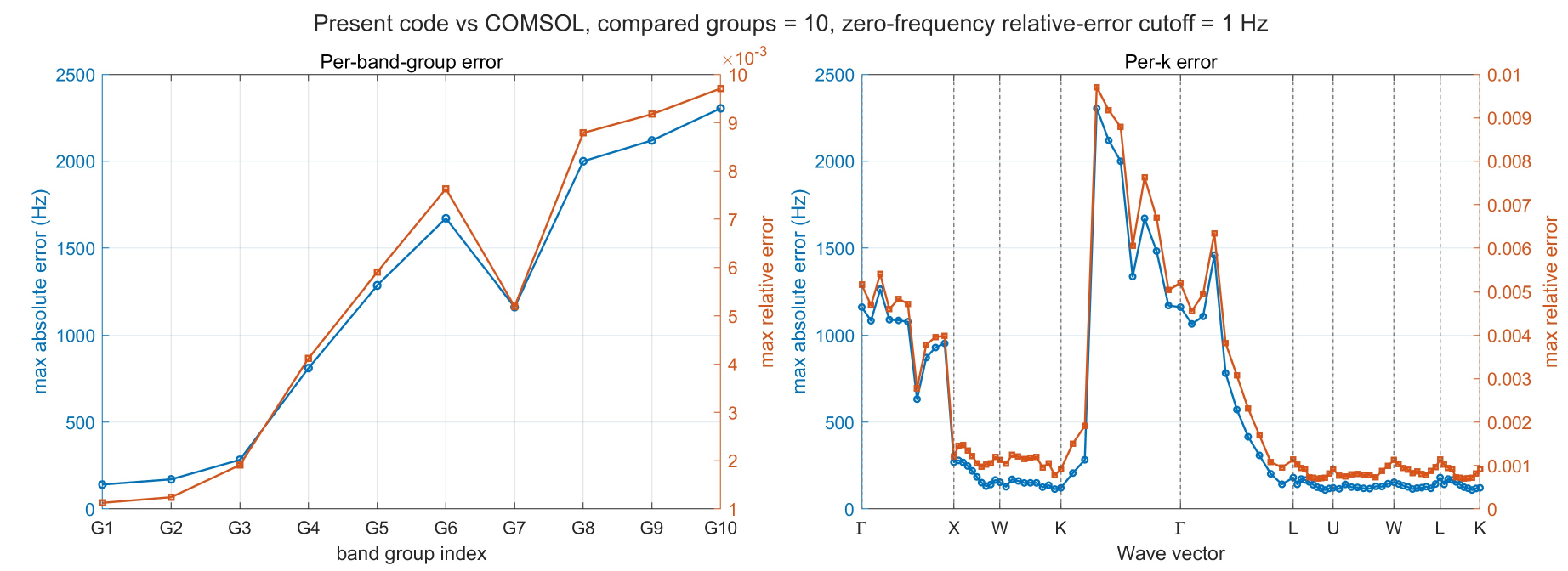}
\caption{Error comparison between the COMSOL reference with \(342459\) finite element degrees of freedom and the grouped \(128^3\) result obtained by the present method. Panel (a) reports the maximum absolute and relative errors by band group, and panel (b) reports the maximum absolute and relative errors along the Bloch path. Relative error statistics exclude the zero-frequency modes at the \(\Gamma\) point.}
\label{fig:code_vs_comsol_error_panels}
\end{figure}

\subsection{Full Space and Block Spectral Agreement}
\label{subsec:num_full_block_error}

To quantify the agreement between the two implementations, the full-space GPU spectrum is used as the internal reference for the same discrete operator. We report the bandwise and per Bloch wave vector errors of the merged block GPU spectrum. Two error measures are used: the maximum absolute frequency difference and the maximum relative frequency difference. Relative errors are computed only on nonzero-frequency branches; the theoretical zero-frequency branches at the \(\Gamma\) point, associated with the zero-derivative modes in Section~\ref{subsec:gamma_deflation}, are excluded from relative error normalization and are not plotted in the relative error curves. This avoids meaningless relative error amplification caused by theoretical zero modes.

Figure~\ref{fig:num_full_block_error_combined} reports the maximum errors by band index and along the Bloch path. In the left panel, for the first \(40\) bands, the maximum absolute error is of order \(10^{-7}\,\mathrm{Hz}\), and the maximum relative error is of order \(10^{-12}\). Thus, under the same material assignment, Bloch path, solver tolerance, and discrete operator, the merged block GPU spectrum is highly consistent with the full-space GPU spectrum. This verifies the algebraic consistency of the block implementation relative to the full-space implementation, rather than the frequency accuracy of the discrete model relative to the continuous physical problem. In the right panel, error peaks mainly occur near the \(\Gamma\) point and some high-symmetry points, where low-frequency zero modes, multiple eigenvalues, or nearly multiple branches make sorting, merging, and relative error normalization more sensitive to small numerical perturbations. After the theoretical zero-frequency modes are excluded, the overall maximum relative error remains at the \(10^{-12}\) level, further confirming that merging the four block subproblems recovers the full-space discrete spectrum.

\begin{figure}[htbp]
\centering
\begin{subfigure}{0.48\textwidth}
\centering
\includegraphics[width=\textwidth]{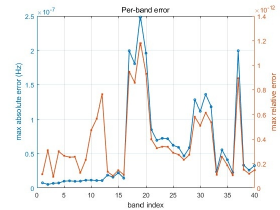}
\caption{Per-band maximum error.}
\label{fig:num_error_by_band}
\end{subfigure}\hfill
\begin{subfigure}{0.48\textwidth}
\centering
\includegraphics[width=\textwidth]{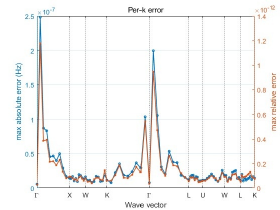}
\caption{Per Bloch wave vector maximum error.}
\label{fig:num_error_by_k}
\end{subfigure}
\caption{Maximum errors between the full-space GPU and block GPU results. Panel (a) shows the bandwise maximum errors, and panel (b) shows the per Bloch wave vector maximum errors. In both panels, the left axis shows the maximum absolute error and the right axis shows the maximum relative error. Relative error statistics exclude the zero-frequency modes at the \(\Gamma\) point.}
\label{fig:num_full_block_error_combined}
\end{figure}

\subsection{Grid Refinement}
\label{subsec:num_band_convergence}

The band convergence test examines the empirical decay of frequency errors under grid refinement. The full-space GPU result on a \(256^3\) grid is used as the high-resolution reference, and the tested grids are \(8^3\), \(16^3\), \(32^3\), \(64^3\), and \(128^3\). All refinement tests use the same continuous spherical inclusion geometry, material parameters, Bloch path, and subpixel volume-fraction assignment rule.

Motivated by the near-fourfold clustering discussed in Section~\ref{subsec:num_full_block_bands}, the first \(40\) bands are grouped by averaging every four consecutive bands, giving \(10\) band groups. This grouping is used as a diagnostic device for the present symmetric benchmark, not as a general band-tracking rule. The error of each group is the average relative error of the nonzero-frequency branches in that group with respect to the \(256^3\) reference result. At the \(\Gamma\) point, theoretical zero-frequency modes are excluded from the relative-error average to avoid meaningless zero-frequency normalization. The horizontal axis is \(h=1/N_g\), and the vertical axis is the grouped average relative error. An \(O(h^2)\) reference slope is added to illustrate the empirical decay trend under grid refinement. Because this grouped metric may average out small splittings inside a local four-band cluster, it is used only as a convergence diagnostic for this benchmark and is not intended to replace branch-resolved band tracking in general configurations.

Figure~\ref{fig:num_band_convergence_gamma_x_w} shows the grouped-band relative errors at three representative high-symmetry wave vectors, \(\Gamma\), \(X\), and \(W\). Under the above grouped-error metric and with the \(256^3\) full-space GPU result as reference, most band groups exhibit an empirical decay close to the \(O(h^2)\) reference slope. This observation should be interpreted as a numerical diagnostic at representative wave vectors, rather than as a general second-order error estimate for heterogeneous anisotropic problems. Deviations in a few groups, especially at the \(\Gamma\) point or near multiple and nearly crossing branches, are mainly associated with zero-mode removal, branch sorting, and relative error normalization.

\begin{figure}[htbp]
\centering
\begin{minipage}{0.32\textwidth}
\centering
\includegraphics[width=\textwidth]{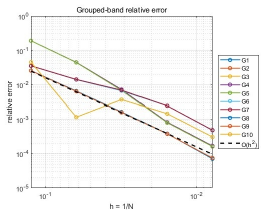}
\par\smallskip
{\small (a) \(\Gamma\) point}
\end{minipage}\hfill
\begin{minipage}{0.32\textwidth}
\centering
\includegraphics[width=\textwidth]{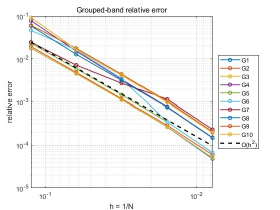}
\par\smallskip
{\small (b) \(X\) point}
\end{minipage}\hfill
\begin{minipage}{0.32\textwidth}
\centering
\includegraphics[width=\textwidth]{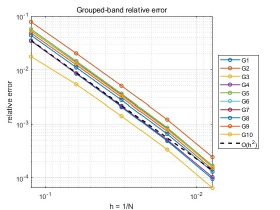}
\par\smallskip
{\small (c) \(W\) point}
\end{minipage}
\caption{Grouped band relative error convergence curves at representative high-symmetry wave vectors: \(\Gamma\), \(X\), and \(W\). \(G1,\ldots,G10\) denote the groups obtained by averaging every four of the first \(40\) bands. The dashed line indicates an \(O(h^2)\) reference slope.}
\label{fig:num_band_convergence_gamma_x_w}
\end{figure}

\subsection{Grid-Refinement Iteration Counts and Runtime}
\label{subsec:num_grid_runtime}

We next report the dependence of the inner iteration count and the running time on grid refinement. This diagnostic separates two effects that are mixed in a single end-to-end timing comparison: the growth of the discrete problem size and the change in the number of Krylov iterations required by the shifted linear systems. The sweep uses the full \(N_k=91\) Bloch path for every grid, the same material assignment, the same shift rule, the same tolerance settings, and the same full-space GPU implementation.

Table~\ref{tab:grid_runtime} shows that the average inner PCG iteration count increases from \(28.66\) on the \(8^3\) grid to about \(34.28\) on the \(128^3\) grid, while no inner PCG failures occur in any run. The wall-clock time increases from \(3715.19\,\mathrm{s}\) to \(34175.14\,\mathrm{s}\), reflecting both the enlarged velocity space and the cost of the full Bloch path calculation. Figure~\ref{fig:grid_runtime_iterations} gives the corresponding visual summary.

\begin{table}[htbp]
\centering
\caption{Grid-refinement iteration counts and running time for the full \(N_k=91\) Bloch path. All runs use the same material assignment, shift rule, tolerance settings, and full-space GPU implementation.}
\label{tab:grid_runtime}
\small
\renewcommand{\arraystretch}{1.12}
\begin{tabular}{cccccc}
\toprule
Grid & Velocity DOFs & \(N_k\) & Avg. PCG iters & Max PCG iters & Wall time (s) \\
\midrule
\(8^3\)   & \(1{,}536\)     & \(91\) & \(28.66\) & \(43\) & \(3715.19\) \\
\(16^3\)  & \(12{,}288\)    & \(91\) & \(32.40\) & \(47\) & \(3933.84\) \\
\(32^3\)  & \(98{,}304\)    & \(91\) & \(34.26\) & \(48\) & \(4617.74\) \\
\(64^3\)  & \(786{,}432\)   & \(91\) & \(34.27\) & \(48\) & \(6389.75\) \\
\(128^3\) & \(6{,}291{,}456\) & \(91\) & \(34.28\) & \(48\) & \(34175.14\) \\
\bottomrule
\end{tabular}
\end{table}

\begin{figure}[htbp]
\centering
\includegraphics[width=0.96\textwidth]{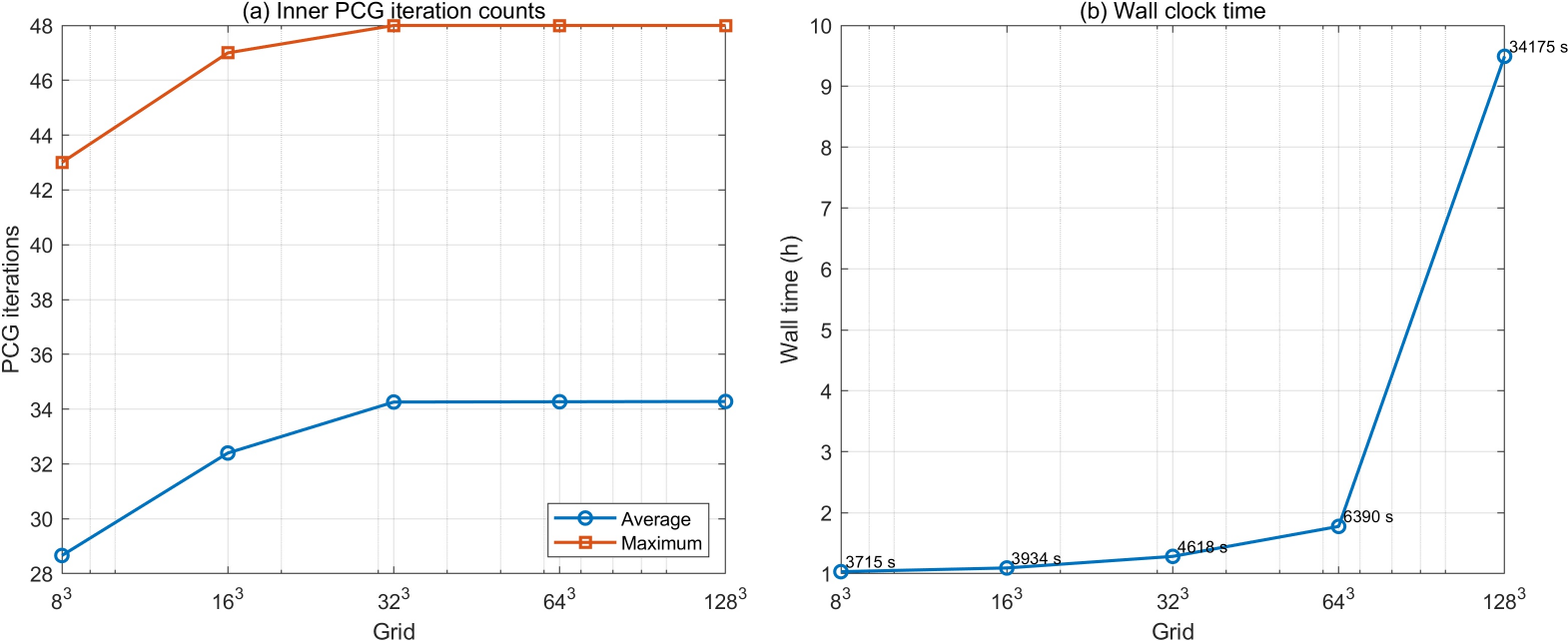}
\caption{Grid-refinement inner PCG iteration counts and wall-clock time for the full \(N_k=91\) Bloch path. Panel (a) shows the average and maximum PCG iteration counts, and panel (b) shows the end-to-end wall-clock time. All runs use the same material assignment, shift rule, tolerance settings, and full-space GPU implementation.}
\label{fig:grid_runtime_iterations}
\end{figure}

\subsection{Wall-Clock Time}
\label{subsec:num_timing}

Table~\ref{tab:num_timing_revised} reports the wall-clock time of one complete band scan on a \(60\times60\times60\) grid with \(N_k=91\) Bloch wave vectors and the first \(40\) bands. All three runs use the material assignment, solver tolerance, and shift-invert settings specified in Section~\ref{subsec:num_setting}. The full-space CPU version takes \(55847\,\mathrm{s}\), the single-GPU full-space version takes \(3912\,\mathrm{s}\), and the four-GPU parallel block version takes \(1709\,\mathrm{s}\). Relative to the full-space CPU version, the single-GPU full-space version and the four-GPU block version achieve end-to-end wall-clock speedups of \(14.27\) and \(32.68\), respectively. Relative to the single-GPU full-space version, the four-GPU block version achieves an end-to-end wall-clock speedup of \(2.29\).

The reported times are complete wall-clock times and include matrix-vector products, inner PCG iterations, outer eigenvalue iterations, \(\Gamma\)-point nullspace treatment, and the necessary data-organization overhead. The implementation does not separately record the preconditioner construction time, outer \texttt{eigs} iteration time, or data-transfer time. Therefore, these results measure the end-to-end performance of the present implementation on the tested hardware. In particular, the comparison between the single-GPU full-space run and the four-GPU block run is a multi-GPU wall-clock comparison, not a hardware-normalized single-GPU speedup. The observed acceleration results from the exact dimension reduction of the four invariant blocks, their parallel solution, and the use of multiple GPU resources.

\begin{table}[htbp]
\centering
\caption{End-to-end running time on a \(60\times60\times60\) grid with \(N_k=91\) and the first \(40\) bands.}
\label{tab:num_timing_revised}
\small
\renewcommand{\arraystretch}{1.12}
\begin{tabular}{lccc}
\toprule
Metric & Full CPU & Full GPU & Block GPU \\
\midrule
Hardware
& CPU
& \(1\times\mathrm{V100}\)
& \(4\times\mathrm{V100}\) \\
Running time
& \(55847\,\mathrm{s}\)
& \(3912\,\mathrm{s}\)
& \(1709\,\mathrm{s}\) \\
Equivalent time
& \(15\) h \(30\) min \(47\) s
& \(1\) h \(5\) min \(12\) s
& \(28\) min \(29\) s \\
Speedup over CPU
& \(1.00\)
& \(14.27\)
& \(32.68\) \\
Speedup over full GPU
& n.a.
& \(1.00\)
& \(2.29\) \\
\bottomrule
\end{tabular}
\end{table}

\section{Conclusion}
\label{sec:conclusion}

This paper has proposed a structured discretization framework based on body-diagonal derivative reconstruction and parity blocking for Bloch-periodic band-structure calculations of 3D generally anisotropic and spatially heterogeneous phononic crystals. The method reconstructs the three lattice coordinate derivatives from phase-shifted differences along four body-diagonal directions and incorporates the coupled derivative terms in a general anisotropic Voigt constitutive law into the discrete divergence-strain structure
\[
A_h(\bk)v_h=\omega^2 M_hv_h,\qquad
A_h(\bk)=B_h(\bk)C_hB_h(\bk)^H.
\]
When the three grid sizes are even and the stiffness and mass matrices enter as nodewise local multiplication matrices, the body-diagonal shifts induce two independent parity invariants, leading to four mutually uncoupled block subspaces. A unified Fourier SVD representation and its weighted extension were formulated for both the full and block spaces, and a practical multi-\(\bk\) band scan workflow was developed using \(\Gamma\)-point nullspace treatment, shift-invert Lanczos iteration, inner PCG solves, and GPU matrix-vector products.

Numerical experiments on a \(60\times60\times60\) grid with \(N_k=91\) Bloch wave vectors and the first \(40\) bands show that the merged block GPU spectrum agrees with the full-space GPU spectrum at the displayed frequency scale. After excluding the zero-frequency modes at the \(\Gamma\) point, the maximum absolute error over bands and wave vectors is of order \(10^{-7}\,\mathrm{Hz}\), and the maximum relative error is of order \(10^{-12}\). This verifies the algebraic consistency between the block and full-space implementations for the same discrete operator. The external COMSOL comparison with the grouped \(128^3\) computed spectrum gives a maximum absolute discrepancy of \(2.304220\times10^3\,\mathrm{Hz}\) and a maximum relative discrepancy of \(9.705655\times10^{-3}\), while the mean relative discrepancy over valid nonzero entries is \(8.775761\times10^{-4}\). The full-space CPU, single-GPU full-space, and four-GPU block implementations require \(15\) h \(30\) min \(47\) s, \(1\) h \(5\) min \(12\) s, and \(28\) min \(29\) s, respectively; the four-GPU block implementation gives an end-to-end wall-clock speedup of about \(2.29\) relative to the single-GPU full-space implementation. This timing comparison reflects the tested hardware allocation, especially the use of four GPUs in the block run, and should not be interpreted as a hardware-normalized single-GPU speedup. Grid-refinement tests using the \(256^3\) full-space GPU result as a reference further show that, under the selected high-symmetry wave vectors and grouped-average error metric, many band groups exhibit empirical decay close to an \(O(h^2)\) reference slope. The grid-refinement sweep from \(8^3\) to \(128^3\) further shows that the average inner PCG iteration count stabilizes at about \(34\) for refined grids, while the wall-clock time increases with the enlarged velocity space.

The applicability of the method is determined by its structural assumptions. The parity blocking relies on even grid sizes in all three directions and on nodewise local multiplication forms of \(C_h\) and \(M_h\). If the material discretization introduces nonlocal cross-node couplings, the commutation relation with the parity operators must be reexamined. For spatially heterogeneous materials, the Fourier SVD representation does not imply complete modewise decoupling of the spectral problem; rather, it provides a structured variable transformation and preconditioning framework for Bloch periodic derivative operators. At the \(\Gamma\) point, in addition to the three physical rigid-translation modes, the even grid body-diagonal differences induce checkerboard-type discrete zero modes, which must be handled by deflation and should not be interpreted as physical rigid-body modes. Future work will focus on higher-order geometric descriptions of curved interfaces, sharper band-tracking near multiple branches, systematic assessment of material assignment rules near curved interfaces, and the possible extension of useful reduced structures to odd grids and more general lattice discretizations.

\section*{CRediT Authorship Contribution Statement}
J. Zhang: Conceptualization, Methodology, Software, Formal analysis, Writing -- original draft. X.-L. Lyu: Software, Validation, Investigation, Visualization, Writing -- review and editing. T. Li: Supervision, Funding acquisition, Writing -- review and editing. W.-W. Lin: Methodology, Supervision, Writing -- review and editing.

\section*{Acknowledgements}
X.-L. Lyu was partially supported by the National Natural Science Foundation of China (NSFC) 12501520 and Basic Research Program of Jiangsu Province BK20251310. T. Li was partially supported by NSFC 12371377 and the Jiangsu Provincial Scientific Research Center of Applied Mathematics under Grant No. BK20233002. This research was funded partially by Shanghai Institute for Mathematics and Interdisciplinary Sciences under grant number SIMIS-ID-2024-LG. We thank Tianhe-2 and the Big Data Computing Center in Southeast University, China, for the use of their computing resources.

\appendix

\section{From the First Order Form to the Elastic Equations}
\label{app:block_verify}

This appendix explains the correspondence between \eqref{eq:first_order_eig_block_main} and the frequency-domain momentum and constitutive equations, and verifies that eliminating the stress variable from the first-order block form recovers the continuous generalized eigenvalue problem \eqref{eq:LEEP_main}.

The operator \(B\) defined in the main text maps Voigt stress variables to the divergence term in the momentum equation. In the Cartesian frame used for the Voigt notation, let
\[
\widetilde{\bm\tau}=
[\tau_{x_1x_1},\tau_{x_2x_2},\tau_{x_3x_3},
\tau_{x_2x_3},\tau_{x_1x_3},\tau_{x_1x_2}]^\top.
\]
Then, up to the common factor \(-\mathrm i\),
\[
B\widetilde{\bm\tau}
=
-\mathrm i
\begin{bmatrix}
\partial_{x_1}\tau_{x_1x_1}
+\partial_{x_2}\tau_{x_1x_2}
+\partial_{x_3}\tau_{x_1x_3}\\
\partial_{x_1}\tau_{x_1x_2}
+\partial_{x_2}\tau_{x_2x_2}
+\partial_{x_3}\tau_{x_2x_3}\\
\partial_{x_1}\tau_{x_1x_3}
+\partial_{x_2}\tau_{x_2x_3}
+\partial_{x_3}\tau_{x_3x_3}
\end{bmatrix}.
\]
Thus \(B\widetilde{\bm\tau}\) represents the Cartesian divergence of the stress tensor in Voigt form. The adjoint \(B^H\), acting on velocity variables, gives the corresponding gradient strain operator under the chosen Hermitian convention. Combined with the local constitutive matrix \(\bm C(\mathbf x)\), this gives the Voigt stress variable through
\[
\widetilde{\bm\tau}=\frac{1}{\omega}\bm C(\mathbf x)B^H\bm v
\]
under the scaling used in the first-order eigenvalue form. Substituting this relation into the velocity equation gives
\[
\rho(\mathbf x)^{-1}B\widetilde{\bm\tau}
=
\omega \bm v,
\]
and hence
\[
B\bm C(\mathbf x)B^H\bm v
=
\omega^2 \rho(\mathbf x)\bm v.
\]
This is precisely \eqref{eq:LEEP_main}. Therefore, the block eigenvalue problem \eqref{eq:first_order_eig_block_main} and the second-order generalized eigenvalue problem \eqref{eq:LEEP_main} have the same nonzero spectral content under the stated velocity-stress scaling.

\section{Diagonalizing the Bloch Shift Operators}
\label{app:K_spectrum}

This appendix records the spectral structure of the Bloch shift matrices used in Section~\ref{sec:blochdisc}. It explains why the 3D shifts \(H_1,H_2,H_3\) can be simultaneously diagonalized.

\subsection{One Dimensional Bloch Shifts}

Consider the one-dimensional Bloch shift
\[
K_\ell=
\begin{pmatrix}
0 & I_{n_\ell-1}\\
e^{\mathrm i\phi_\ell} & 0
\end{pmatrix}.
\]
It satisfies \(K_\ell^{n_\ell}=e^{\mathrm i\phi_\ell}I\). Therefore, its eigenvalues are
\[
\lambda_{\ell,p}
=
\exp\!\left(\mathrm i\frac{\phi_\ell+2\pi p}{n_\ell}\right),
\qquad p=0,\ldots,n_\ell-1.
\]
An associated normalized eigenvector is
\[
x_{\ell,p}
=
\frac{1}{\sqrt{n_\ell}}
(1,\lambda_{\ell,p},\lambda_{\ell,p}^2,\ldots,\lambda_{\ell,p}^{n_\ell-1})^\top .
\]
Let
\[
X_\ell=(x_{\ell,0},x_{\ell,1},\ldots,x_{\ell,n_\ell-1}).
\]
Then \(X_\ell\) is unitary and
\[
K_\ell=X_\ell\Lambda_\ell X_\ell^H,
\qquad
\Lambda_\ell=\operatorname{diag}(\lambda_{\ell,0},\ldots,\lambda_{\ell,n_\ell-1}).
\]

\subsection{Tensor Product Diagonalization}

With
\[
T=X_3\otimes X_2\otimes X_1,
\]
the 3D shift matrices in \eqref{eq:H_def_main} satisfy
\[
H_1=T(\,I_{n_3}\otimes I_{n_2}\otimes\Lambda_1\,)T^H,
\]
\[
H_2=T(\,I_{n_3}\otimes \Lambda_2\otimes I_{n_1}\,)T^H,
\qquad
H_3=T(\,\Lambda_3\otimes I_{n_2}\otimes I_{n_1}\,)T^H.
\]
Since the three shifts commute, all body-diagonal shifts and difference operators constructed from them inherit the same common diagonalization. This provides the algebraic basis for the fast operator actions discussed in Section~\ref{sec:impl}.

\section{\texorpdfstring{\(\Gamma\)-Point Zero Modes}{Gamma Point Zero Modes}}
\label{app:gamma_null}

This appendix supplements the conclusion on the \(\Gamma\)-point zero modes in Section~\ref{sec:impl}.

In a standard periodic derivative discretization, constant vector fields give three rigid translation zero modes. For the even grid body-diagonal differences used in this paper, however, the four body-diagonal shifts preserve the following four scalar modes:
\[
\begin{aligned}
\chi_0&=1,
&\chi_{12}&=(-1)^{i+j},\\
\chi_{13}&=(-1)^{i+k},
&\chi_{23}&=(-1)^{j+k}.
\end{aligned}
\]
Therefore,
\[
D_{d_\ell}\chi_\mu=0,\qquad
\ell=1,2,3,4,\quad
\mu\in\{0,12,13,23\}.
\]
Multiplying each scalar zero-derivative mode by the three velocity components gives
\[
\dim\ker B_h(\Gamma)^H=12.
\]
Here \(\chi_0\mathbf e_\alpha\) corresponds to the physical rigid translation mode, whereas the other three scalar modes correspond to additional checkerboard-type discrete zero modes induced by the even grid body-diagonal differences.

\bibliographystyle{abbrvnat}
\bibliography{refs}

@article{SigalasEconomou1992,
  author  = {Sigalas, M. M. and Economou, E. N.},
  title   = {Elastic and acoustic wave band structure},
  journal = {Journal of Sound and Vibration},
  volume  = {158},
  number  = {2},
  pages   = {377--382},
  year    = {1992},
  doi     = {10.1016/0022-460X(92)90059-7}
}

@article{Kushwaha1993,
  author  = {Kushwaha, M. S. and Halevi, P. and Dobrzynski, L. and Djafari-Rouhani, B.},
  title   = {Acoustic band structure of periodic elastic composites},
  journal = {Physical Review Letters},
  volume  = {71},
  number  = {13},
  pages   = {2022--2025},
  year    = {1993},
  doi     = {10.1103/PhysRevLett.71.2022}
}

@article{Hussein2014,
  author  = {Hussein, Mahmoud I. and Leamy, Michael J. and Ruzzene, Massimo},
  title   = {Dynamics of phononic materials and structures: Historical origins, recent progress, and future outlook},
  journal = {Applied Mechanics Reviews},
  volume  = {66},
  number  = {4},
  pages   = {040802},
  year    = {2014},
  doi     = {10.1115/1.4026911}
}

@article{Laude2009,
  author  = {Laude, Vincent and Achaoui, Younes and Benchabane, Sarah and Khelif, Abdelkrim},
  title   = {Evanescent {Bloch} waves and the complex band structure of phononic crystals},
  journal = {Physical Review B},
  volume  = {80},
  pages   = {092301},
  year    = {2009},
  doi     = {10.1103/PhysRevB.80.092301}
}

@article{LyuTianLiLin2024JSC,
  author  = {Lyu, Xing-Long and Tian, Heng and Li, Tiexiang and Lin, Wen-Wei},
  title   = {Fast {SVD}-Based Linear Elastic Eigenvalue Problem Solver for Band Structures of {3D} Phononic Crystals},
  journal = {Journal of Scientific Computing},
  volume  = {99},
  pages   = {20},
  year    = {2024},
  doi     = {10.1007/s10915-024-02483-8}
}

@article{Wu2004PRB,
  author  = {Wu, Tsung-Tsong and Huang, Zi-Gui and Lin, Shih},
  title   = {Surface and bulk acoustic waves in two-dimensional phononic crystals consisting of materials with general anisotropy},
  journal = {Physical Review B},
  volume  = {69},
  pages   = {094301},
  year    = {2004},
  doi     = {10.1103/PhysRevB.69.094301}
}

@article{WuHsuHuang2005,
  author  = {Wu, Tsung-Tsong and Hsu, Zhen-Cheng and Huang, Zi-Gui},
  title   = {Band gaps and the electromechanical coupling coefficient of a surface acoustic wave in a two-dimensional piezoelectric phononic crystal},
  journal = {Physical Review B},
  volume  = {71},
  pages   = {064303},
  year    = {2005},
  doi     = {10.1103/PhysRevB.71.064303}
}

@article{Wang2021APM,
  author  = {Wang, Liqun and Zheng, Hui and Zhao, Meiling and Shi, Liwei and Hou, Songming},
  title   = {{Petrov--Galerkin} method for the band structure computation of anisotropic and piezoelectric phononic crystals},
  journal = {Applied Mathematical Modelling},
  volume  = {89},
  pages   = {1090--1105},
  year    = {2021},
  doi     = {10.1016/j.apm.2020.08.026}
}

@article{Yan2025JCP,
  author  = {Yan, Jiaxin and Wang, Liqun and Zhang, Yifan and Zhao, Meiling and Shi, Liwei},
  title   = {A {FEM} towards {3D} multi-component elastic interface problems and phononic crystals with nested and intersected scatterer geometries},
  journal = {Journal of Computational Physics},
  volume  = {534},
  pages   = {114017},
  year    = {2025},
  doi     = {10.1016/j.jcp.2025.114017}
}

@article{Virieux1984,
  author  = {Virieux, Jean},
  title   = {{SH}-wave propagation in heterogeneous media: Velocity--stress finite-difference method},
  journal = {Exploration Geophysics},
  volume  = {15},
  number  = {4},
  pages   = {265--276},
  year    = {1984},
  doi     = {10.1071/EG984265A}
}

@article{Virieux1986,
  author  = {Virieux, Jean},
  title   = {{P-SV} wave propagation in heterogeneous media: Velocity--stress finite-difference method},
  journal = {Geophysics},
  volume  = {51},
  number  = {4},
  pages   = {889--901},
  year    = {1986},
  doi     = {10.1190/1.1442147}
}

@article{Saenger2000,
  author  = {Saenger, Erik H. and Gold, Norbert and Shapiro, Serge A.},
  title   = {Modeling the propagation of elastic waves using a modified finite-difference grid},
  journal = {Wave Motion},
  volume  = {31},
  number  = {1},
  pages   = {77--92},
  year    = {2000},
  doi     = {10.1016/S0165-2125(99)00023-2}
}

@article{Saenger2004,
  author  = {Saenger, Erik H. and Bohlen, Thomas},
  title   = {Finite-difference modeling of viscoelastic and anisotropic wave propagation using the rotated staggered grid},
  journal = {Geophysics},
  volume  = {69},
  number  = {2},
  pages   = {583--591},
  year    = {2004},
  doi     = {10.1190/1.1707078}
}

@article{BansalSen2008,
  author  = {Bansal, Reeshidev and Sen, Mrinal K.},
  title   = {Finite-difference modelling of {S}-wave splitting in anisotropic media},
  journal = {Geophysical Prospecting},
  volume  = {56},
  number  = {3},
  pages   = {293--312},
  year    = {2008},
  doi     = {10.1111/j.1365-2478.2007.00693.x}
}

@article{Bernth2011,
  author  = {Bernth, Henrik and Chapman, Chris H.},
  title   = {A comparison of the dispersion relations for anisotropic elastodynamic finite-difference grids},
  journal = {Geophysics},
  volume  = {76},
  number  = {3},
  pages   = {WA43--WA50},
  year    = {2011},
  doi     = {10.1190/1.3555067}
}

@article{Yang2015,
  author  = {Yang, L. and Yan, H. and Liu, H.},
  title   = {Optimal rotated staggered-grid finite-difference schemes for elastic wave modeling in {TTI} media},
  journal = {Journal of Applied Geophysics},
  volume  = {122},
  pages   = {40--52},
  year    = {2015},
  doi     = {10.1016/j.jappgeo.2015.08.007}
}

@article{Gao2017,
  author  = {Gao, Kai and Huang, Lianjie},
  title   = {An improved rotated staggered-grid finite-difference method with fourth-order temporal accuracy for elastic-wave modeling in anisotropic media},
  journal = {Journal of Computational Physics},
  volume  = {350},
  pages   = {361--386},
  year    = {2017},
  doi     = {10.1016/j.jcp.2017.08.053}
}

@article{Chen2006,
  author  = {Chen, Hao and Wang, Xiuming and Zhao, Haibo},
  title   = {A rotated staggered grid finite-difference with the absorbing boundary condition of a perfectly matched layer},
  journal = {Chinese Science Bulletin},
  volume  = {51},
  number  = {19},
  pages   = {2304--2314},
  year    = {2006},
  doi     = {10.1007/s11434-006-2127-8}
}

@article{Zhang2012,
  author  = {Zhang, Wei and Shen, Yang and Zhao, Li},
  title   = {Three-dimensional anisotropic seismic wave modelling in spherical coordinates by a collocated-grid finite-difference method},
  journal = {Geophysical Journal International},
  volume  = {188},
  number  = {3},
  pages   = {1359--1381},
  year    = {2012},
  doi     = {10.1111/j.1365-246X.2011.05331.x}
}

@article{Lisitsa2010,
  author  = {Lisitsa, Vadim and Vishnevskiy, Dmitriy},
  title   = {{Lebedev} scheme for the numerical simulation of wave propagation in 3D anisotropic elasticity},
  journal = {Geophysical Prospecting},
  volume  = {58},
  number  = {4},
  pages   = {619--635},
  year    = {2010},
  doi     = {10.1111/j.1365-2478.2009.00862.x}
}

@article{Koene2021,
  author  = {Koene, Erik F. M. and Robertsson, Johan O. A. and Andersson, Fredrik},
  title   = {Anisotropic elastic finite-difference modeling of sources and receivers on {Lebedev} grids},
  journal = {Geophysics},
  volume  = {86},
  number  = {2},
  pages   = {A21--A27},
  year    = {2021},
  doi     = {10.1190/geo2020-0522.1}
}

@article{PabstGregorova2013,
  author  = {Pabst, Willi and Gregorov{\'a}, Eva},
  title   = {Elastic properties of silica polymorphs -- a review},
  journal = {Ceramics--Silik{\'a}ty},
  volume  = {57},
  number  = {3},
  pages   = {167--184},
  year    = {2013},
  url     = {https://www.ceramics-silikaty.cz/2013/pdf/2013_03_167.pdf}
}

@misc{IOFFESilicon,
  author       = {{Ioffe Institute}},
  title        = {Silicon: Mechanical properties, elastic constants, lattice vibrations},
  howpublished = {Ioffe Institute, New Semiconductor Materials database},
  year         = {2026},
  note         = {Accessed 2026-05-20},
  url          = {https://www.ioffe.ru/SVA/NSM/Semicond/Si/mechanic.html}
}

\end{document}